\documentclass[11pt]{article}
\usepackage[OT2,T1]{fontenc}
\usepackage[utf8]{inputenc}
\usepackage[top = 3cm, bottom  = 3cm, left = 3.7cm, right = 3.2cm]{geometry}
\usepackage{amsmath, amscd, amssymb, amsthm, amsfonts, euscript, mathtools, tikz-cd}
\usetikzlibrary{babel}
\usepackage{adjustbox}
\usepackage{changepage}
\usepackage{url}
\usepackage{hyperref}
\usepackage{cite}
\hypersetup{ colorlinks=true, citecolor= black,linkcolor = black}

\usepackage{todonotes}
\usepackage{etoolbox}

\tikzcdset{
  cells={font=\everymath\expandafter{\the\everymath\displaystyle}},
}

\DeclareSymbolFont{cyrletters}{OT2}{wncyr}{m}{n}
\DeclareMathSymbol{\Sha}{\mathalpha}{cyrletters}{"58}
\DeclareMathSymbol{\Bcyr}{\mathalpha}{cyrletters}{"42}

\DeclareMathOperator{\Spec}{\mathrm{Spec}}

\theoremstyle{definition}
\newtheorem{thm}{Theorem}[section]

\newtheorem*{defi*}{Definition}
\newtheorem{prop}[thm]{Proposition}
\newtheorem*{prop*}{Proposition}
\newtheorem{lemma}[thm]{Lemma}
\newtheorem{cor}[thm]{Corollary}
\newtheorem*{cor*}{Corollary}
\newtheorem{rmk}[thm]{Remark}

\newtheorem{thmalpha}{Theorem}

\title{Kato-Kuzumaki's properties for function fields over higher local fields}
\author{ Felipe Gambardella \\ \vspace*{-1ex} \small \textit{Centre de Mathematiques Laurent Schwartz - Ecole Polytechnique} \vspace*{1ex}\\ \texttt{felipe.gambardella@polytechnique.edu}}
\date{}
\begin{document}
\maketitle

\begin{abstract}
    Let $k$ be a $d$-local field such that the corresponding $1$-local field $k^{(d-1)}$ is a $p$-adic field and $C$ a curve over $k$. Let $K$ be the function field of $C$. We prove that for each $n,m \in \mathbf{N}$, and hypersurface $Z$ of $\mathbf{P}^n_K$ with degree $m$ such that $m^{d+1} \leq n$, the $(d+1)$-th Milnor $\mathrm{K}$-theory group is generated by the images norms of finite extension $L$ of $K$ such that $Z$ admits an $L$-point. Let $j \in \{1,\cdots , d\}$. When $C$ admits a point in an extension $l/k$ that is not $i$-ramified for every $i \in \{1, \cdots, d-j\}$ we generalise this result to hypersurfaces $Z$ of $\mathbf{P}_K^n$ with degree $m$ such that $m^{j+1} \leq n$. \par
    In order to prove these results we give a description of the Tate-Shafarevich group $\Sha^{d+2}(K,\mathbf{Q}/\mathbf{Z}(d+1))$ in terms of the combinatorics of the special fibre of certain models of the curve $C$.
\end{abstract}

\section{Introduction}
        In \cite{KK1986DimensionOfFields} Kato and Kuzumaki proposed a set of conjectures that had as objective giving a Diophantine chatacterisation of the cohomological dimension of a field. These conjectures are stated using a variant of the $C_i$ property involving Milnor $\mathrm{K}$-theory and hypersurfaces of small degree. They expected that these properties would chatacterise the fields of small cohomological dimension. \par 
        More precisely, fix a pair of integers $i,q \geq 0$ and a field $k$. For any finite extension $l/k$ Milnor $\mathrm{K}$-theory admits a norm map $\mathrm{N}_{l/k}: \mathrm{K}_q(l) \to \mathrm{K}_q(k)$, see \cite[\S 7.3]{GilleSzamuely2017CentralSimple}. For a $k$-variety $Z$ Kato and Kuzumaki introduced its norm group $\mathrm{N}_q(Z/k)$ generated by $\mathrm{N}_{l/k}(\mathrm{K}_q(l))$ where $l/k$ ranges over all finite extensions such that $Z(l) \neq \emptyset$. They introduced the $C_i^q$ property: for every pair of integers $m, n\geq 1$, finite extension $k'/k$, and hypersurface $Z \subseteq \mathbf{P}^n_{k'}$ of degree $m$ with $m^i \leq n$ we have $\mathrm{K}_q(k') = \mathrm{N}_q(Z/k')$. For example, a field $k$ satisfies $C_i^0$ if and only if for every pair of integers $m,n \geq 1$, field extension $k'/k$, and hypersurface $Z$ in $\mathbf{P}^n_{k'}$ of degree $m$ such that $m^i \leq n$ the variety $Z$ admits a $0$-cycle of degree $1$. In the other extreme, a field $k$ satisfies $C_0^q$ if and only if for every tower $k''/k'/k$ of finite extensions the norm map $\mathrm{N}_{k''/k'}: \mathrm{K}_{q}(k'') \to \mathrm{K}_q(k')$ is surjective. \par
        Kato-Kuzumaki's conjecture states that for every triple of non-negative integers $n,i$ and $q$ such that $i+q =n$, a field $k$ has cohomological dimension at most $n$ if and only if it satisfies $C_i^q$.  In one direction this conjecture generalises Serre's question about whether the cohomological dimension of a $C_i$ field is at most $i$ \cite{Serre1965CohomologieGaloisienne}. In the original paper \cite{KK1986DimensionOfFields} the authors already pointed out that Bloch-Kato's conjecture (proved by Rost and Voevodsky \cite{Riou2014BlochKato}) implies that a field $k$ has cohomological dimension at most $q$ if and only if it satisfies $C^q_0$.\par 
        Unfortunately, there are counter-examples to the conjectures. In \cite{Merkurjev1991NotC20} a field of cohomological dimension $2$ that does not satisfy $C_2^0$ is constructed, and in \cite{CTMadore2004NotC10} we can find a characteristic zero field of cohomological dimension $1$ that does not satisfy $C_1^0$. Since both counterexamples are constructed through transfinite induction, it is reasonable to consider that Kato-Kuzumaki's conjecture is still open for fields naturally arising in arithmetic geometry. \par
        To the knowledge of the author the only known cases of the full Kato-Kuzumaki's conjecture are the finite field extensions of the fields $\mathbf{C}(x_1, \cdots, x_n)$ \cite[Theorem 2.2]{Diego2018ConjectureKatoKuzumaki}, $\mathbf{C}(x_1, \cdots, x_n)(\!(t)\!)$ \cite[Theorem 3.9]{Diego2018ConjectureKatoKuzumaki}, or $\mathbf{C}(\!(x)\!)(\!(y)\!)$ \cite[Example 4.8]{Wittenberg2015KKQp}. Some other particular cases are known, for example Wittenberg proved that $\mathbf{Q}_p$ and totally imaginary number fields satisfy $C_1^1$ in \cite[Corollaire 5.5 Théorème 6.1]{Wittenberg2015KKQp}. Moreover, a stronger variant of the $C_1^{d}$ property is known for $d$-local fields (such as $\mathbf{Q}_p(\!(t_2)\!) \cdots (\!(t_d)\!)$ or $\mathbf{F}_p(\!(t_1)\!) \cdots (\!(t_d)\!)$) if we ignore the residue characteristic, see \cite[Théorème 4.2 and Corollaire 4.7]{Wittenberg2015KKQp}. \par
        In \cite[\S 2]{Diego2018ConjectureKatoKuzumaki} we can find a different proof of the fact that totally imaginary number fields satisfy the $C_1^1$ property. The proof in \cite[\S 2]{Diego2018ConjectureKatoKuzumaki} relies on local-global principles for certain torsors of normic tori. This local-global approach was also fruitful for function fields of curves over $p$-adic fields. In \cite{DiegoLuco2024KKp-adicFunction} the authors exploit the arithmetic duality theorems for motivic cohomology over such fields (see \cite[Théorème 0.1]{Diego2015LGPrincipleHigherLocal}) to prove that function fields of $p$-adic curves satisfy the $C_2^2$ property \cite[Main Theorem 1]{DiegoLuco2024KKp-adicFunction}. Moreover, some partial results towards the $C_1^2$ property are also obtained, see \cite[Main Theorem 2]{DiegoLuco2024KKp-adicFunction}. Those arithmetic duality theorems are also valid over function fields of curves over fields of the form $k_0(\!(t_2)\!) \cdots (\!(t_d)\!)$ where $k_0$ is a $p$-adic field. In this article we extend the methods of \cite{DiegoLuco2024KKp-adicFunction} to such function fields. We obtain the following result which is analogous to \cite[Main Theorem 1]{DiegoLuco2024KKp-adicFunction}.
        \begin{thmalpha}\label{thm A}
            The field $\mathbf{Q}_p(\!(t_2)\!) \cdots (\!(t_d)\!)(x)$ (and its finite extensions) satisfies $C_{d+1}^{d+1}$.
        \end{thmalpha}
        \noindent The proof follows a similar structure as the proof of \cite[Main Theorem 1]{DiegoLuco2024KKp-adicFunction}. The main technical addition is the computation of a certain Tate-Shafarevich group. More precisely, let $k$ be a field of the form $k_0(\!(t_2)\!) \cdots (\!(t_d)\!)$ where $k_0$ is a $p$-adic field and $C$ a smooth projective geometrically integral curve over $k$. Denote by $K$ the function field of $C$. We obtain the following generalisation of \cite[Corollary 2.9]{Kato1986Hasse2dim}
        \begin{thmalpha}\label{thm B}
         There is an isomorphism $\Sha^{d+2}(K,\mathbf{Q}/\mathbf{Z}(d+1)) \simeq \left(\mathbf{Q}/\mathbf{Z}\right)^{r(C)}$ for some integer $r(C)$ that can be explicitly described in terms of $C$ (see Section \ref{sec preliminary} for a precise definition). Moreover, this isomorphism is stable under purely unramified extension of the base field, i.e. extensions of the form $l_0K/K$ where $l_0/k_0$ is a finite unramified extension.
        \end{thmalpha} 
        \noindent With Kato-Kuzumaki's conjecture in mind, one expects the field $K$ to satisfy $C_1^{d+1}$. In this direction we prove the following result.
        \begin{thmalpha}\label{thm C}
            Let $j \in \{0 ,\cdots, d\}$. Let $l/k$ be a finite extension such that for every $i \in \{1, \cdots d-j\}$ it is not $j$-ramified (see Section \ref{sec preliminary}) and $C(l) \neq \emptyset$. Then for every pair of integers $m,n \geq 1$ and hypersurface $Z$ in $\mathbf{P}^n_K$ of degree $m$ such that $m^{j+1} \leq n$ we have $\mathrm{N}_{d+1}(Z/K) = \mathrm{K}_{d+1}(K)$.
        \end{thmalpha}
        \noindent This theorem can be thought as a generalisation of \cite[Main Theorem 2]{DiegoLuco2024KKp-adicFunction}. In particular, the case $j=0$ gives the following corollary.
        \begin{cor*}
            Assume that $C$ admits a rational point. Then for every pair of integers $m,n\geq 1$ and hypersurface $Z$ in $\mathbf{P}^n_K$ of degree $m$ such that $m\leq n$ we have $\mathrm{N}_{d+1}(Z/K) = \mathrm{K}_{d+1}(K)$.
        \end{cor*}
\section{Preliminaries and notation} \label{sec preliminary}
\subsection*{Milnor \texorpdfstring{$K$}{K}-Theory}
    Let $k$ be a field and $q$ a non-negative integer. We define the $q$-th Milnor $\mathrm{K}$-group of $k$ as follows: for $q=0$ we set $\mathrm{K}_0(k) = \mathbf{Z}$ and for $q \geq 1$ we set 
    \[\mathrm{K}_q(k) := \left(k^{\times} \right)^{\otimes q} \; \big/ \; \left\langle a_1 \otimes \cdots \otimes a_q \, | \, \exists i,j, \; i \neq j,\; a_i + a_j = 1 \right\rangle\]
    where the tensor product is over $\mathbf{Z}$. For $a_1, \cdots, a_q \in k^{\times}$ we denote by $\{a_1, \cdots, a_q\}$ the image of $a_1 \otimes \cdots \otimes a_q$ in $\mathrm{K}_q(k)$. Elements of this form are called \textit{symbols}. For any pair of non-negative integers $p,q$ there is a natural pairing
    \[\{\cdot, \cdot\}: \mathrm{K}_p(k) \times \mathrm{K}_q(k) \to \mathrm{K}_{p+q}(k)\]
    induced by the tensor pairing $\left(k^{\times} \right)^{\otimes p} \times \left(k^{\times} \right)^{\otimes q} \to \left(k^{\times} \right)^{\otimes (p+q)}$. \par
    Let $l/k$ be a finite extension. One can construct a norm homomorphism
    \[\mathrm{N}_{l/k}: \mathrm{K}_q(l) \to \mathrm{K}_q(k)\]
    satisfying the following properties
    \begin{itemize}
        \item For $q = 0$, the map $\mathrm{N}_{l/k}: \mathbf{Z}\to \mathbf{Z}$ is given by multiplication by $[l:k]$,
        \item For $q = 1$ the map $\mathrm{N}_{l/k}: \mathrm{K}_1(l) \to \mathrm{K}_1(k)$ coincides with the usual norm $l^{\times} \to k^{\times}$,
        \item For any pair of non-negative integers $p,q$, we have $\mathrm{N}_{l/k} (\{\alpha, \beta\}) = \{\alpha, \mathrm{N}_{l/k}(\beta) \}$ for $\alpha \in \mathrm{K}_p(k)$  and $\beta \in \mathrm{K}_q(l)$, and
        \item If $m$ is a finite extension of $l$ we have $\mathrm{N}_{m/k} = \mathrm{N}_{l/k} \circ \mathrm{N}_{m/l}$.
    \end{itemize}
    The construction can be found in \cite[Section 1.7]{Kato1980GeneralizationClassFieldK2} or \cite[Section 7.3]{GilleSzamuely2017CentralSimple}. \par 
    Let $K$ be a Henselian discrete valuation field with ring of integers $R$, maximal ideal $\mathfrak{m}$ and residue field $k$. Then for every strictly positive integer $q$ there exists a unique \textit{residue map}
    \[ \partial : \mathrm{K}_q(K) \to \mathrm{K}_{q-1}(k)\]
    such that for every uniformiser $\pi$ and units $u_2,\cdots u_q \in R^{\times}$ we have
    \[\partial(\{\pi, u_2, \cdots , u_q \}) = \{\overline{u_2}, \cdots , \overline{u_q}\}\]
    where $\overline{u_2}, \cdots , \overline{u_q}$ represent the images of $u_2,\cdots u_q$ in the residue field. \par
    We denote the kernel of $\partial: \mathrm{K}_q(K) \to \mathrm{K}_{q-1}(k)$ by $\mathrm{U}_q(K)$. It is generated by symbols of the form $\{u_1 , \cdots ,u_q \}$ where $u_1, \cdots ,u_q$ are units in $R$. One can define a specialisation map $s: \mathrm{U}_q(K) \to \mathrm{K}_{q}(k)$ characterised by $s(\{u_1 , \cdots ,u_q \}) = \{\overline{u_1}, \cdots , \overline{u_q}\}$, see \cite[Proposition 7.1.4]{GilleSzamuely2017CentralSimple}. Denote by $\mathrm{U}_q^1(K)$ the kernel of $s$. It is generated by symbols of the form $\{u, a_2, \cdots, a_q\}$ where $u$ belongs to $1 + \mathfrak{m}$ and $a_2,\cdots, a_q\in K^{\times}$, see \cite[Proposition 7.1.7]{GilleSzamuely2017CentralSimple}. Note that for every prime $\ell$ different from the characteristic of $k$ the group $\mathrm{U}_q^1(K)$ is $\ell$-divisible. \par
    Moreover, the residue map and specialisation map are compatible with the norm map. Indeed, if $L/K$ is a finite extension with ramification index $e$ and the residue field of $L$ is $l$, we have the following commutative diagrams 
    \begin{eqnarray*}
        \begin{tikzcd}
            \mathrm{K}_q(L)/ \mathrm{U}_q(L) \ar[r, "\sim" ', "\partial_L"]  \ar[d,"\mathrm{N}_{L/K}"]& \mathrm{K}_{q-1}(l) \ar[d,"\mathrm{N}_{l/k}"] & \mathrm{U}_q(L)/ \mathrm{U}_q^1(L) \ar[r, "\sim" ', "s_L"]  \ar[d,"\mathrm{N}_{L/K}"]& \mathrm{K}_{q}(l) \ar[d,"e \cdot \mathrm{N}_{l/k}"] \\
            \mathrm{K}_q(K)/ \mathrm{U}_q(K) \ar[r, "\sim" ', "\partial_K"] & \mathrm{K}_{q-1}(k)  & \mathrm{U}_q(K)/ \mathrm{U}_q^1(K) \ar[r, "\sim" ', "s_L"]  & \mathrm{K}_{q}(k)
        \end{tikzcd}
    \end{eqnarray*}
\subsection*{\texorpdfstring{$C_i^q$}{Kato Kuzumaki} properties}
    Let $k$ be a field and $i,q$ non-negative integers. Let $X$ be a $k$-scheme of finite type. We define the $q$\textit{-th norm group} of $X$ as 
    \[\mathrm{N}_q(X/k) = \left\langle \mathrm{N}_{k(x)/k}(\mathrm{K}_q(k(x))) \; | \; x \in X_{(0)}\right\rangle \subseteq \mathrm{K}_q(k)\]
    where $X_{(0)}$ denotes the set of closed points of $X$. When $l/k$ is a finite extension we write $\mathrm{N}_q(l/k)$ for $\mathrm{N}_{q}(\Spec (l) /k)$. The field $k$ is said to have the $C_i^q$ property if for every non-negative integers $n$ and $d$, finite extension $l/k$ and hypersurface $Z$ of $\mathbf{P}^n_l$ of degree $d$ with $d^i \leq n$ we have $\mathrm{N}_q(Z/l) = \mathrm{K}_q(l)$. \par 
    As it was already pointed out in the original article \cite{KK1986DimensionOfFields}, the Bloch-Kato conjecture (now the norm-residue isomorphism) the $C_0^q$ property is equivalent to having cohomological dimension at most $q$.
\subsection*{Motivic complexes}
    Let $k$ be a field and $i$ a non-negative integer. We denote by $z^i(k, -)$ Bloch's cycle complex defined in \cite{Bloch1986AlgCyclesKtheory}. We denote by $\mathbf{Z}(i)$ the étale motivic complex defined by the complex $z^i(k,-)[2i]$. The main properties that we use are the following.
    \begin{enumerate}
        \item[(1)] \textit{Nesterensko-Suslin-Totaro theorem:} There is a natural identification between $\mathrm{H}^i(k,\mathbf{Z}(i))$ and $\mathrm{K}_i(k)$.
        \item[(2)] \textit{Beilinson-Lichtenbaum conjecture:} $\mathrm{H}^{i+1}(k,\mathbf{Z}(i)) = 0$.
    \end{enumerate}
    The Nesterensko-Suslin-Totaro theorem was originally proved in \cite{NesterenkoSuslin1990, Totaro1992MilnorKandAlgK}. It can also be found in \cite[Theorem 5.1]{LectureNotesonMotivicCoh}. The Beilinson-Lichtenbaum conjecture is deduced from the Bloch-Kato conjecture in \cite{SuslinJoukhovitski2006NormVar ,GeisserLevine2000Kincharp}. The Bloch-Kato conjecture was proved in \cite{SuslinJoukhovitski2006NormVar,Voevodsky2011MotivicCohFiniteCoeff}. The implication from the Bloch-Kato conjecture to the Beilinson-Lichtenbaum conjecture can also be found in \cite[Lemma 1.6 and Theorem 1.7]{HaesemeyerWeibel2019NormResidueiso}.
\subsection*{Bloch-Ogus-Kato complexes}
    Let $\mathcal{X}$ be an excellent scheme and $m \in  \mathbf{Z}$ invertible in $\mathcal{X}$. Fix $n, s \in \mathbf{N}$. In this article we use two complexes associated to $\mathcal{X}$. The first one is the Kato complex constructed in \cite[\S 1]{Kato1986Hasse2dim}. The term in degree $j$ of this complex is $\bigoplus_{v\in\mathcal{X}^{(j)}} \mathrm{H}^{n-j}(\kappa(v), \mathbf{Z}/m(s-j))$ where $\mathcal{X}^{(j)}$ are the points of codimension $j$ of $\mathcal{X}$ and $\kappa(v)$ is the residue field of $\mathcal{X}$ at $v$. The differential
    \[ \bigoplus_{v\in\mathcal{X}^{(j)}} \mathrm{H}^{n-j}(\kappa(v), \mathbf{Z}/m(s-j))\xrightarrow{\partial} \bigoplus_{v\in\mathcal{X}^{(j+1)}} \mathrm{H}^{n-j-1}(\kappa(x), \mathbf{Z}/m(s-j-1)) \]
    are obtained by noramalisation and sum over the usual cohomological residues. For the details on the construction see \cite[\S 1]{Kato1986Hasse2dim} \par
    The second complex of interest is the Bloch-Ogus complex. The degree $j$ term of this complex is given by $\bigoplus_{v\in\mathcal{X}^{(j)}} \mathrm{H}^{n+j}_v(\mathcal{X}, \mathbf{Z}/m(s))$ where $\mathrm{H}^{n+j}_v(\mathcal{X}, \mathbf{Z}/m(s))$ is the colimit of the étale cohomology groups with support $\mathrm{H}^{n+j}_{\overline{\{v\}} \cap U}(U, \mathbf{Z}/m(s))$ when $U$ ranges over all open subsets of $\mathcal{X}$ containing $v$. The differentials
    \[ \bigoplus_{v\in\mathcal{X}^{(j)}} \mathrm{H}^{n+j}_v(\mathcal{X}, \mathbf{Z}/m(s))\xrightarrow{\partial} \bigoplus_{v\in\mathcal{X}^{(j+1)}} \mathrm{H}^{n+j+1}(\kappa(x), \mathbf{Z}/m(s))\]
    are obtained via the localiastion exact sequence and restriction. For more details on the construction see \cite[\S 1]{CTHooblerKahn1997BOGTheorem}. \par
    When the scheme $\mathcal{X}$ is regular, these two complexes get identify up to sign via the Gysin morphism \cite[Theorem 2.5.4]{JansenSaitoSato2014}. It is also possible to say something when the integer $m$ is not invertible in $\mathcal{X}$, see \cite[\S 3 and \S 4]{JansenSaitoSato2014}.
\subsection*{Higher local fields and function fields}
    Let $d$ be a non-negative integer. In this paper we define $0$-local fields to be finite fields and a $d$\textit{-local field} to be a complete discrete valuation field whose residue field is a $(d-1)$-local field. \par
    Let $k$ be a $d$-local field. For $i =0, \cdots , d$ we denote by $k^{(i)}$ the $i$-th residue field. Note that $k^{(i)}$ is a $(d-i)$ local field.  We denote by $\mathcal{O}_k^{(i)}$ the rank $i$ ring of integers of $k$, see \cite[Definition 3.3]{Morrow2012Higherlocal}. We omit the superindex i the case $i=d$, so $\mathcal{O}_k := \mathcal{O}_k^{(d)}$. This is a henselian valuation ring with residue field $k^{(d)}$ and fraction field $k$. \par
    Let $l/k$ be a finite extension, and $j \in \{ 0, \cdots , d \} $. We denote by $[l:k]_j$ the degree of the residual extension $l^{(d-j)}/k^{(d-j)}$ of $j$-local fields. If $j$ is non-zero, $e_j(l/k)$ denotes the ramification index of $l^{(d-j)}/k^{(d-j)}$, and $e_{\leq j}(l/k) = \prod_{i=1}^j e_i(l/k)$. In order to make the notation uniform, we also use the notation $e_0(l/k)$ for $[l:k]_0$. The extension $l/k$ is said to be 
    \begin{itemize}
    \item \textit{purely unramified} if for every $i \in \{1, \cdots, d\}$ we have $e_i(l/k) = 1$, or equivalently $[l:k] = [l:k]_0$,
    \item \textit{purely ramified} if $[l:k]_0=1$,
    \item $j$-\textit{ramified} if $e_j(l/k) \neq 1$,
    \item \textit{totally} $j$-\textit{ramified} if $e_j(l/k) = [l:k]$.
    \end{itemize}
    A system of parameters for $k$ is a sequence $t_1 , \cdots t_n \in k$ such that $t_i$ belongs to $\mathcal{O}^{(1)}_k$ for every $i \in \{1,\cdots ,d\}$, the element $t_n$ is a uniformiser for $\mathcal{O}^{(1)}_k$ and the classes $\overline{t_1} \cdots \overline{t_{d-1}} \in k^{(1)}$ form a system of parameters for $k^{(1)}$. \par 
    \noindent \textbf{We fix the following notation for the rest of the present article:} $k$ denotes a $d$-local field such that the $1$-local field $k^{(d-1)}$ is a $p$-adic field, $C$ is a smooth projective geometrically integral curve over $k$, and $K$ is the function field of $C$. \par
    For $j \in \{ 0, \cdots , d \}$ we define the $j$-th index of $C$ as 
    \[i_{\leq j}(C) = \mathrm{gcd}( \,[k':k]_j \, | \, C(k') \neq \emptyset \, ).\]
    In the case $j=0$ we denote $i_{\leq 0}(C)$ by $i_0(C)$. Similarly, we define the $j$-th ramification index of $C$ as 
    \[i_{\leq j}^{ram}(C) = \mathrm{gcd}( \, e_{\leq j}(k'/k) \, | \, C(k') \neq \emptyset \, ).\]
    \subsection*{Models of curves}
    
    Let $\mathcal{C}$ be a regular, projective flat model of $C$ over $\mathcal{O}_k^{(1)}$ such that the special fibre $Y$ is a strict normal crossing divisor. We define a bipartite graph $\Gamma(\mathcal{C}) = (V,E)$ called the \textit{reduction graph} of $\mathcal{C}$ where
    \begin{itemize}
        \item \textit{Vertices:} $V= Y^{(0)} \sqcup I(Y)$ where $Y^{(0)}$ is the set of generic points of $Y$ and $I(Y)$ is the set of closed points of $Y$ that belongs to more than one component.
        \item \textit{Edges:} $E$ is the set of pairs $\{v,  x\}$ where $v \in Y^{(0)}$, $x \in I(Y)$ and $x \in \overline{\{v\}}$.
    \end{itemize}
    Note that this graph is finite. We denote by $g(C)$ the genus of the reduction graph $\Gamma(\mathcal{C})$. This is independent of the chosen model because for two different models, the reduction graphs are homopoty equivalent, see \cite[Remark 6.1 (b)]{HHK2015LGtorsorsarithmeticcurves}. We define $r(\mathcal{C})$ by induction on $d$ as follows: for $d=0$ we set $r(C)=0$. For $d \geq 1$, let $D_1, \cdots, D_n$ be the irreducible components of the special fibre of $\mathcal{C}$. Note that for every $i \in \{1, \cdots, n\}$ the curve $D_i$ is defined over a $(d-1)$-local field. We set
    \[ r(C) := g(C) +\sum_{i=1}^n r(D_i).\]
    \begin{rmk} \label{rmk invariance rC under tot nr}
        Let $l/k$ be a finite purely unramified extension. Then $r(C_l) = r(C)$. Indeed, this can be checked by induction on $d$: if we choose a model $\mathcal{C}$ to compute $r(C)$, the scheme $\mathcal{C} \otimes_{\mathcal{O}_k}\mathcal{O}_{l}$ is still a suitable model to compute $g(C)$ and by induction for every irreducible component $D'$ of the special fibre of $\mathcal{C}$ we have $r(D_i) = r(D_{i,l_{d-1}})$.
    \end{rmk}
\subsection*{Poitou-Tate duality and reciprocity}
The field $K$ admits a \textit{higher reciprocity law}. We state the result in the following proposition for future reference.
\begin{prop} \label{prop higher recirpocity}
    For every $v \in C^{(1)}$ and $m \in \mathbf{Z}$ we have natural an isomorphisms $j_{m,v}: \mathrm{H}^{d+2}(K_v, \mathbf{Z}/m(d+1)) \to \mathbf{Z}/m$. We have an exact sequence
    \[ \mathrm{H}^{d+2}(K, \mathbf{Q}/ \mathbf{Z}(d+1)) \to \bigoplus_{v \in C^{(1)}} \mathrm{H}^{d+2}(K_v, \mathbf{Q}/ \mathbf{Z}(d+1)) \xrightarrow{\sum_{v\in C^{(1)}} j_v} \mathbf{Q}/ \mathbf{Z} \to 0 \]
    where $j_v$ correspond to the isomorphism $\mathrm{H}^{d+2}(K_v, \mathbf{Q}/\mathbf{Z}(d+1)) \to \mathbf{Q}/\mathbf{Z}$ obtain by taking the colimit of $j_{m,v}$ over $m$.
\end{prop}
\noindent The isomorphism can be found in \cite[Proposition 1.2]{ManhLinh2024ArithmeticsHomo} and the exact sequence can be directly deduced from Proposition \cite[Proposition 2.6]{Diego2015LGPrincipleHigherLocal}. \par
Let $i$ be a non-negative integer and $\mathcal{C}$ a regular flat projective model of $C$ over $\mathcal{O}^{(1)}_k$. For a Galois module $M$ we define the Tate-Shafarevich groups as follows
\begin{align*}
    \Sha^i(K,M) & := \ker\left( \mathrm{H}^i(K,M) \to \prod_{v \in C^{(1)}} \mathrm{H}^i(K_v, M) \right) \\
    \Sha^i_{\mathcal{C}}(K, M) &:=\ker\left( \mathrm{H}^i(K,M) \to \prod_{v \in \mathcal{C}^{(1)}} \mathrm{H}^i(K_v, M) \right).
\end{align*}
We recall the Poitou-Tate duality for motivic complexes \cite[Théorème 0.1]{Diego2015LGPrincipleHigherLocal}. Let $\hat{T}$ be a finitely generated Galois module that is free as an abelian group. Set $\check{T} := \mathrm{Hom}_K (\hat{T} , \mathbf{Z})$ and $T:= \check{T} \otimes_K^{\mathbf{L}} \mathbf{Z}(d+1)$. Then there is a perfect pairing of finite groups 
\[ \overline{\Sha^2(K,\hat{T})} \times \Sha^{d+2}(K,T) \to \mathbf{Q} / \mathbf{Z}\]
where $\overline{\Sha^2(K,\hat{T})}$ denotes the quotient of $\Sha^2(K,\hat{T})$ by the maximal divisible subgroup . \par
Note that in the case $\hat{T} = \mathbf{Z}$ the Beilinson-Lichtenbaum conjecture together with Poitou-Tate duality imply that $\Sha^2(K,\mathbf{Z})$ is divisible. Moreover, by Shapiro's lemma, the group $\Sha^2(K,\mathbf{Z}[E/K])$ is also divisible for any étale algebra $E/K$.

\section{Cohomological computations}\label{sec coh compu}
    The objective of this section is to compute the group $\Sha^{d+2}(K, \mathbf{Z}/m(d+1))$ for every $m \in \mathbf{Z}$. We fix a projective regular flat model $\mathcal{C}$ of $C$ such that the special fibre $Y$ is a strict normal crossing divisor. The strategy is to reduce the computation to the corresponding Tate-Shafarevich group of the components of the special fibre $Y$ of $\mathcal{C}$ and proceed by induction.
    \begin{prop} \label{prop sha as kernel of residues}
        The residues maps induce an isomorphism
        \[ \Sha^{d+2}(K, \mathbf{Z}/m(d+1)) \simeq \ker \left( \bigoplus_{v \in Y^{(0)}} \mathrm{H}^{d+1}(\kappa(v), \mathbf{Z}/m(d)) \to \bigoplus_{x \in Y^{(1)}} \mathrm{H}^{d}(\kappa(x), \mathbf{Z}/m(d-1)) \right).\]
    \end{prop}
    \begin{proof}
        For ease of notation, we denote by $F(r)$ the Galois module $\mathbf{Z}/m(r)$ for every $r \in \mathbf{Z}$. For every $v \in Y^{(0)}$ we denote by $\partial_v:\mathrm{H}^{d+2}(K, F(d+1)) \to \mathrm{H}^{d+1}(\kappa(v), F(d))$ the cohomological residue. Note that $\partial_v$ factorises as 
        \[ \mathrm{H}^{d+2}(K, F(d+1)) \to \mathrm{H}^{d+2}(K_v, F(d+1)) \to \mathrm{H}^{d+1}(\kappa(v), F(d))\] 
        and the second arrow is injective because its kernel is isomorphic to $\mathrm{H}^{d+2}(\mathcal{O}_v, F(d+1)) \simeq \mathrm{H}^{d+2}(\kappa(v), F(d+1))$ which is trivial because $\kappa(v)$ has cohomological dimension $d+1$. Then the map
        \begin{equation}\label{eq residues injective} 
            \bigoplus_{v \in Y^{(0)}} \partial_{v}:  \Sha^{d+2}(K, F(d+1)) \to \bigoplus_{v \in Y^{(0)}} \mathrm{H}^{d+1}(\kappa(v), F(d)) 
        \end{equation}
        is injective thanks to \cite[Theorem 3.3.6]{HHK2014LGGaloisCohomology}. Moreover, the Kato complex of $\mathcal{C}$ gives rise to an anti-commutative square
        \begin{equation}\label{diag residue square Bloch-Ogus}
            \begin{tikzcd}
                \mathrm{H}^{d+2}(K,F(d+1)) \ar[d] \ar[r] & \bigoplus_{v \in C^{(1)}} \mathrm{H}^{d+1}(\kappa(v),F(d)) \ar[d] \\
                \bigoplus_{v \in Y^{(0)}} \mathrm{H}^{d+1}(\kappa(v),F(d)) \ar[r] & \bigoplus_{x \in Y^{(1)}} \mathrm{H}^{d}(\kappa(x),F(d-1)).
            \end{tikzcd}
        \end{equation}
        Which proves that the image of \eqref{eq residues injective} lies in 
        \[
            \ker \left( \bigoplus_{v \in Y^{(0)}} \mathrm{H}^{d+1}(\kappa(v), \mathbf{Z}/m(d)) \to \bigoplus_{x \in Y^{(1)}} \mathrm{H}^{d}(\kappa(x), \mathbf{Z}/m(d-1)) \right).
        \]
        Giving the inclusion of $\Sha^{d+2}(K, F(d+1))$ in the desired kernel. In order to deduce surjectivity we proceed as in \cite[Proposition 5.2]{Kato1986Hasse2dim}. As it was recalled in section \ref{sec preliminary}, the Kato complex gets identify with the Bloch-Ogus complex, \cite[Theorem 2.5.4]{JansenSaitoSato2014}. Under this identification the the anti-commutative square \eqref{diag residue square Bloch-Ogus} gets identified with the (anti-commutative) rightmost square of the following diagram
        \begin{equation}\label{diag Cousin vs localisation}
            \begin{tikzcd}[column sep = 3ex, row sep =2ex]
                \mathrm{H}^{d+2}(C, F(d+1)) \ar[d, twoheadrightarrow, "\delta"] \ar[r,"f"] & \mathrm{H}_{\eta}^{d+2}(C,F(d+1)) \ar[r] \ar[d] & \bigoplus_{v \in C^{(1)}} \mathrm{H}^{d+3}_v(C, F(d+1))  \ar[d]\\
                \mathrm{H}^{d+3}_{Y}(\mathcal{C}, F(d+1)) \ar[r,"g"] & \bigoplus_{v \in Y^{(0)}}\mathrm{H}_{v}^{d+3}(\mathcal{C},F(d+1)) \ar[r] & \bigoplus_{x \in Y^{(1)}} \mathrm{H}^{d+4}_x(\mathcal{C}, F(d+1)).
            \end{tikzcd}
        \end{equation}
        Note that the leftmost square is commutative.
        \begin{lemma}
            The morphism $\delta$ is surjective.
        \end{lemma}
        \begin{proof}
            The localisation exact sequence gives us an exact sequence
            \[\mathrm{H}^{d+2}(C, F(d+1)) \to \mathrm{H}^{d+3}_{Y}(\mathcal{C}, F(d+1)) \to \mathrm{H}^{d+3}(\mathcal{C}, F(d+1)). \]
            Then it is enough to prove that $\mathrm{H}^{d+3}(\mathcal{C}, F(d+1))$ is trivial. Denote by $j: \Spec K \to \mathcal{C}$ the inclusion of the generic point. The Leray spectral sequence \cite[Theorem 1.18 (a)]{Milne1980etalecoh}
            \[\mathrm{H}^p(\mathcal{C},\mathrm{R}^qj_*(F(d+1))) \Rightarrow \mathrm{H}^{p+q}(K, F(d+1))\]
            identifies $\mathrm{H}^{d+3}(\mathcal{C}, F(d+1))$ with a subgroup of $\mathrm{H}^{d+3}(K, F(d+1))$. The later group is trivial because the cohomological dimension of $K$ is at most $d+2$.
        \end{proof}
        \noindent The rows of \eqref{diag Cousin vs localisation} are exact. Indeed, let $U$ be an open subset of $C$. The localisation exact sequence takes the following form
        \begin{equation*} 
            \cdots \to \mathrm{H}^{d+2}(C, F(d+1)) \to \mathrm{H}^{d+2}(U, F(d+1)) \to \mathrm{H}^{d+3}_{C \setminus U}(C, F(d+1)) \to \cdots.
        \end{equation*}
        Passing to the limit over $U$ while taking into account the isomorphism 
        \[ \mathrm{H}^{d+3}_{C \setminus U}(C, F(d+1))  \simeq \bigoplus_{v \in C \setminus U} \mathrm{H}^{d+3}_v(C, F(d+1))\]
        gives the exact sequence
        \begin{equation*} 
            \cdots \to \mathrm{H}^{d+2}(C, F(d+1)) \to \mathrm{H}^{d+2}(K, F(d+1)) \to \bigoplus_{v \in C^{(1)}} \mathrm{H}^{d+3}_v(C, F(d+1)) \to \cdots.
        \end{equation*}
        Let $V \subseteq \mathcal{C}$ be an open subset. The localisation exact sequence \cite[Remark III 1.26]{Milne1980etalecoh} associated with the triple $(\mathcal{C},V \cup C, C)$ take the following form
        \[\cdots \to \mathrm{H}^{d+3}_{Y}(\mathcal{C}, F(d+1)) \to \mathrm{H}^{d+3}_{Y \cap V}(V, F(d+1)) \to \mathrm{H}^{d+4}_{Y \setminus V}(\mathcal{X}, F(d+1))\to \cdots \]
        Taking the limit over $V$ while taking into account the \cite[Lemma 1.2.1]{CTHooblerKahn1997BOGTheorem} and
        \[ \mathrm{H}^{d+3}_{Y \setminus V}(\mathcal{C}, F(d+1))  \simeq \bigoplus_{v \in Y \setminus V} \mathrm{H}^{d+3}_v(\mathcal{C}, F(d+1))\]
        gives the desired exact sequence. \par
        Let $\alpha \in \ker \left( \bigoplus_{v \in Y^{(0)}} \mathrm{H}^{d+1}(\kappa(v), \mathbf{Z}/m(d)) \to \bigoplus_{x \in Y^{(1)}} \mathrm{H}^{d}(\kappa(x), \mathbf{Z}/m(d-1)) \right)$. Since the bottom row of \eqref{diag Cousin vs localisation} is exact and $\delta$ is surjective we can find a class $\beta \in \mathrm{H}^{d+2}(C, F(d+1))$ such that $g(\delta(\beta)) = \alpha$. We deduce that $\alpha$ lifts to $f(\beta)$ which belongs to $\Sha^{d+2}(K,F(d+1))$ because the top row of \eqref{diag Cousin vs localisation} is a complex. 
    \end{proof}
        Note that for every $v \in Y^{(0)}$, the residue field $\kappa(v)$ is a one-variable function field over a $(d-1)$-local field. Then it makes sense to talk about the Tate-Shafarevich group $\Sha^{d+1}(\kappa(v), \mathbf{Z}/m(d))$ with respect to valuations coming from closed points of $\overline{\{v \}} \subseteq \mathcal{C}$. The subgroup $\bigoplus_{v \in Y^{(0)}} \Sha^{d+1}(\kappa(v), \mathbf{Z}/m(d))$ of $\bigoplus_{v \in Y^{(0)}} \mathrm{H}^{d+1}(\kappa(v), \mathbf{Z}/m(d))$ is contained in 
        \[\ker \left( \bigoplus_{v \in Y^{(0)}} \mathrm{H}^{d+1}(\kappa(v), \mathbf{Z}/m(d)) \to \bigoplus_{x \in Y^{(1)}} \mathrm{H}^{d}(\kappa(x), \mathbf{Z}/m(d-1)) \right).\]
        Under the identifications in Proposition \ref{prop sha as kernel of residues}, we can see $\bigoplus_{v \in Y^{(0)}} \Sha^{d+1}(\kappa(v), \mathbf{Z}/m(d))$ as a subgroup of $\Sha^{d+2}(K, \mathbf{Z}/m(d+1))$. In order to simplify the notation, we set
        \[
            A := \Sha^{d+2}(K, \mathbf{Z}/m(d+1)) \big/ \bigoplus_{v \in Y^{(0)}} \Sha^{d+1}(\kappa(v), \mathbf{Z}/m(d)).
        \]
    \begin{prop}\label{prop Sha modulo the previous Shas}
        Let $g$ be the genus of the reduction graph of $\mathcal{C}$. Then we have a natural isomorphism $A \simeq (\mathbf{Z}/m)^g$.
    \end{prop}
    \begin{proof}
        Let $\Gamma(\mathcal{C}) =(V,E)$ be the reduction graph of $\mathcal{C}$. Fix 
        \[B' := \bigoplus_{v \in Y^{(0)}}  \mathrm{H}^{d+1}(\kappa(v), \mathbf{Z}/m(d)) \, \big/ \, \Sha^{d+1}(\kappa(v), \mathbf{Z}/m(d)).\]
        Let $B$ be the subgroup of $B'$ formed by classes represented by $(\alpha_v)_{v \in Y^{(0)}}$ such that $\partial_x(\alpha) = 0$ for every $x \not \in I(Y)$. Note that $A$ is a subgroup of $B$. \par 
        For every $\{x,v\} \in E$ we denote by $\kappa(v)_x$ the completion of $\kappa(v)$ with respect to the valuation induced by $x$. From Proposition \ref{prop higher recirpocity} we deduce the following exact sequence
        \begin{equation}\label{seq B as kernel 1}
            0 \to B \to\bigoplus_{\{x,v\} \in E} \mathrm{H}^{d+1}(\kappa(v)_x, \mathbf{Z}/m(d)) \xrightarrow{\sum j_{x,m}}\bigoplus_{v \in Y^{(0)}} \mathbf{Z}/m \to 0
        \end{equation}
        Moreover Proposition \ref{prop higher recirpocity} gives the following identifications
        \begin{align}\label{eq identifications invariants}
            \begin{split}
            \mathrm{H}^{d}(\kappa(x), \mathbf{Z}/m(d-1)) & \simeq \mathbf{Z}/m \\
            \mathrm{H}^{d+1}(\kappa(v)_x, \mathbf{Z}/m(d)) & \simeq \mathbf{Z}/m
            \end{split}
        \end{align}
       This isomorphism together with the exact sequence \eqref{seq B as kernel 1} identify $B$ with the kernel of the sum map $\Sigma_1: \bigoplus_{\{x,v \} \in E} \mathbf{Z}/m \to \bigoplus_{v \in Y^{(0)}} \mathbf{Z}/m$. \par
       On the other hand, Proposition \ref{prop sha as kernel of residues} identifies $A$ with the kernel of the map 
       \[ \partial: B \to \bigoplus_{x \in I(Y)} \mathrm{H}^{d}(\kappa(x), \mathbf{Z}/m(d-1))\]
       induced by the differential of the Kato complex. Let $\Sigma_2:\bigoplus_{\{x,v \} \in E} \mathbf{Z}/m \to \bigoplus_{v \in I(Y)} \mathbf{Z}/m$ be the map given by $\Sigma_2((a_{\{x,v\}})_{\{x,v\}\in E}) := (\sum_{x \in \overline{v}}a_{\{x,v\}})_{x \in I(Y)}$. Under the identification \eqref{eq identifications invariants} we get the following commutative diagram 
        \begin{equation}
            \begin{tikzcd}
                B \ar[r] \ar[d,"\partial"] & \bigoplus_{\{x,v\} \in E} \mathbf{Z}/m  \ar[d, "\Sigma_2"]  \\
                \bigoplus_{x \in I(Y)} \mathbf{Z}/m \ar[r, equal] & \bigoplus_{x \in I(Y)} \mathbf{Z}/m  \\
            \end{tikzcd}
        \end{equation}
        Putting all of this together we have identified $A$ with the subgroup $\ker \Sigma_1 \cap \ker \Sigma_2$ of $\bigoplus_{\{x,v\} \in E} \mathbf{Z}/m$. We conclude the desired isomorphism applying \cite[Lemme 1.2 (iii)]{Diego2019Duality2dim}
    \end{proof}
    
    \begin{thm}\label{thm computation of Sha}
        The group $\Sha^{d+2}(K, \mathbf{Z}/m(d+1))$ is isomorphic to $(\mathbf{Z}/m)^{r(C)}$.
    \end{thm}
    \begin{proof}
        We proceed by induction on $d$. The case $d=1$ is treated in \cite[Corollary 2.9]{Kato1986Hasse2dim}. For the general case, we can apply Proposition \ref{prop Sha modulo the previous Shas} to deduce an exact sequence of $\mathbf{Z}/m$-modules
        \[
            0 \to \bigoplus_{v \in Y^{(0)}} \Sha^{d+1}(\kappa(v), \mathbf{Z}/m(d)) \to \Sha^{d+2}(K, \mathbf{Z}/m(d+1)) \to \left(\mathbf{Z}/m\right)^g \to 0.
        \]
        By induction hypothesis $\Sha^{d+1}(\kappa(v), \mathbf{Z}/m(d))$ is isomorphic to $\left(\mathbf{Z}/m\right)^{r(\overline{\{v \}})}$ for every $v \in Y^{(0)}$. We conclude because the previous exact sequence splits.
    \end{proof}
    \section{The \texorpdfstring{$C_{d+1}^{d+1}$}{Kato Kuzumaki} property for function fields over \texorpdfstring{$d$}{d}-local fields}
    
    In this section we follow the strategy of proof for \cite[Main Theorem 1]{DiegoLuco2024KKp-adicFunction} to obtain the following analogue result.
    \begin{thm}\label{thm chi2 torsion}
        Let $l/k$ be a finite purely unramified extension and set $L:= lK$. Let $Z$ be a proper $K$-variety. Then the quotient
        \[ \mathrm{K}_{d+1}(K) \, \big/ \, \left\langle \mathrm{N}_{d+1}(L/K), \mathrm{N}_{d+1}(Z/K) \right\rangle\]
        is $\chi_{K}(Z,E)^2$-torsion for any coherent sheaf $E$ on $Z$.
    \end{thm}
    Here, $\chi_K(Z,E)$ is the Euler characteristic of $E$.
    \subsection*{Step 0: Interpreting norms in Milnor \texorpdfstring{$\mathrm{K}$}{K}-theory in terms of motivic cohomology}
    The same proof of \cite[Lemma 3.2]{DiegoLuco2024KKp-adicFunction} gives the following lemma.
    \begin{lemma}\label{lemma norms Ktheory and motivic}
        Let $L$ be a field and $L_1, \cdots, L_r$ be finite separable extensions of $L$. Consider the étale $L$-algebra $E:= \prod_{i=1}^r L_i$ and the Galois module $\check{T}$ defined by the following exact sequence
        \[0 \to \check{T} \to \mathbf{Z} [E/L] \to \mathbf{Z} \to 0.\]
        Then for every $n \in \mathbf{N}$ there is a natural isomorphism
        \[ \mathrm{H}^{n+1}(L, \check{T} \otimes \mathbf{Z}(n)) \simeq \mathrm{K}_{n}(L) \, \big/ \, \left\langle \mathrm{N}_{n}(L_i/L)\; | \; i \in \{1 ,\cdots r \}\right\rangle.\]
    \end{lemma}

    \subsection*{Step 1: Reducing to curves with 0th residual index $1$}
    In this step we prove the following proposition.
    \begin{prop}\label{prop reduction to index 1}
        Let $l/k$ be a purely unramified extension such that $[l:k]$ divides $i_{0}(C)$. Fix $L := lK$. Then the norm map $\mathrm{N}_{L/K}: \mathrm{K}_{d+1}(L) \to \mathrm{K}_{d+1}(K)$ is surjective.
    \end{prop}
    \begin{proof}
        We proceed by induction on $d$. The case $d=1$ is treated in \cite[Proposition 3.3]{DiegoLuco2024KKp-adicFunction}. Consider the Galois module $\check{T}$ defined by the following exact sequence
        \[0 \to \check{T} \to \mathbf{Z}[L/K] \to \mathbf{Z} \to 0.\]
        Since the extension $L/K$ is cyclic, we can also fit $\check{T}$ in the following exact sequence
        \[ 0 \to \mathbf{Z} \to \mathbf{Z}[L/K] \to \check{T} \to 0. \]
        This gives rise to a distinguished triangle 
        \[ \mathbf{Z}(d+1) \to \mathbf{Z}[L/K] \otimes \mathbf{Z}(d+1) \to \check{T} \otimes \mathbf{Z}(d+1) \to \mathbf{Z}(d+1)[1].\]
        By the Beilinson-Lichtenbaum conjectures, the group $\mathrm{H}^{d+2}(L, \mathbf{Z}(d+1))$ is trivial. Hence, the natural arrow
        \[\Sha_{\mathcal{C}}^{d+2}(K,\check{T} \otimes \mathbf{Z}(d+1)) \to \Sha_{\mathcal{C}}^{d+3}(K,\mathbf{Z}(d+1))\]
        is injective, where $\mathcal{C}$ is a regular proper flat model of $C$ whose reduced special fibre $C_0$ is a strict normal crossing divisor. We deduce from the distinguished triangle
        \[\mathbf{Z}(d+1) \to \mathbf{Q}(d+1) \to \mathbf{Q} / \mathbf{Z} (d+1) \to \mathbf{Z}(d+1)[1]\]
        and the vanishing of the groups $\mathrm{H}^{d+2}(K,\mathbf{Q}(d+1))$ and $\mathrm{H}^{d+3}(K,\mathbf{Q}(d+1))$, which follows from \cite[Lemma 2.5 and Theorem 2.6 (c)]{Kahn2012ClassesMotivicEtaleCycles}, an isomorphism
        \[\Sha_{\mathcal{C}}^{d+2}(K, \mathbf{Q} / \mathbf{Z} (d+1)) \simeq \Sha_{\mathcal{C}}^{d+3}(K,\mathbf{Z} (d+1)).\]
        \begin{lemma}
            The group $\Sha_{\mathcal{C}}^{d+2}(K,\mathbf{Q}/\mathbf{Z} (d+1))$ is trivial.
        \end{lemma}
        \begin{proof}
            Let $\alpha \in \Sha_{\mathcal{C}}^{d+2}(K,\mathbf{Q}/\mathbf{Z} (d+1))$. There exists an $m$ such that $\alpha$ comes from $\alpha_m \in  \mathrm{H}^{d+2}(K,\mathbf{Z}/m (d+1))$. Since $d+2$ is the cohomological dimension of $K$ there are finitely many places $v$ such that the image of $\alpha_{m}$ in $\mathrm{H}^{d+2}(K_v, \mathrm{Z}/m(d+1))$ are non-zero. After enlarging $m$ we can assume that $\alpha_m $ belongs to $\Sha_{\mathcal{C}}^{d+2}(K, \mathbf{Z}/m (d+1))$. We conclude because $\Sha_{\mathcal{C}}^{d+2}(K, \mathbf{Z}/m (d+1))$ is trivial thanks to \cite[Theorem 3.3.6]{HHK2014LGGaloisCohomology}.
        \end{proof}
        Then, $\Sha_{\mathcal{C}}^{d+2}(K,\check{T} \otimes \mathbf{Z}(d+1))$ is also trivial. Note that $L/K$ is totally split at every $v \in C^{(1)}$. Then given the identifications in Lemma \ref{lemma norms Ktheory and motivic} we get an injection
        \[ \mathrm{K}_{d+1}(K) / \mathrm{im}(\mathrm{N}_{L/K}) \hookrightarrow \prod_{v \in \mathcal{C}^{(1)} \setminus C^{(1)}} \mathrm{K}_{d+1}(K_v) / \mathrm{im}(\mathrm{N}_{L\otimes K_v/K_v}). \]
        Then it is enough to prove that $\mathrm{K}_{d+1}(K_v) / \mathrm{im}(\mathrm{N}_{L\otimes K_v/K_v})$ is trivial for every $v \in  \mathcal{C}^{(1)} \setminus C^{(1)}$. \par
        Fix $v \in  \mathcal{C}^{(1)} \setminus C^{(1)}$. The residue map induces an exact sequence
        \begin{equation} \label{exseq residue index}
            0 \to \frac{\mathrm{U}_{d+1}(K_v)}{\mathrm{U}_{d+1}(K_v) \cap  \mathrm{im}(\mathrm{N}_{L \otimes K_v/K_v})} \to \frac{\mathrm{K}_{d+1}(K_v)}{ \mathrm{im}(\mathrm{N}_{L\otimes K_v/K_v})} \to \frac{\mathrm{K}_{d}(k^{(1)}(C_v))}{ \mathrm{im}(\mathrm{N}_{l^{(1)}(C_v)/k^{(1)}(C_v)})} \to 0
        \end{equation}
        where $C_v$ is the irreducible component of $C_0$ corresponding to $v$. We can apply \cite[Lemma 4.6]{Wittenberg2015KKQp} to deduce that $[l: k]_{d-1}$ divides $i_0(C_v)$. Applying the induction hypothesis we know that the rightmost group in \eqref{exseq residue index} is trivial. We are reduced to prove that $\mathrm{U}_{d+1}(K_v) \subseteq  \mathrm{im}(\mathrm{N}_{L\otimes K_v/K_v})$. Since the extension $L/K$ is unramified at $v$, we get the following commutative diagram with exact rows
        \begin{equation}\label{exseq norms and unit groups index}
            \begin{tikzcd}
                0 \ar[r] & \mathrm{U}_{d+1}^1(L\otimes K_v) \ar[r] \ar[d] & \mathrm{U}_{d+1}(L\otimes K_v) \ar[r] \ar[d, "\mathrm{N}_{L\otimes K_v/K_v}"] & \mathrm{K}_{d+1}(l^{(1)}(C_v))^{\oplus [L:K]} \ar[r] \ar[d, "\oplus \mathrm{N}_{l^{(1)}(C_v)/k^{(1)}(C_v)}"] & 0 \\
                0 \ar[r] & \mathrm{U}_{d+1}^1(K_v) \ar[r]  & \mathrm{U}_{d+1}(K_v) \ar[r]  & \mathrm{K}_{d+1}(k^{(1)}(C_v)) \ar[r] & 0.
            \end{tikzcd}
        \end{equation}
        The leftmost vertical arrow of \eqref{exseq norms and unit groups index} is surjective because $\mathrm{U}_{d}^1(K_v)$ is divisible. The rightmost vertical arrows of \eqref{exseq norms and unit groups index} is also surjective because $k^{(1)}(C_v)$ has the $C_{0}^{d+1}$ property. We deduce that $\mathrm{U}_{d+1}(L\otimes K_v) \to \mathrm{U}_{d+1}(K_v)$ is also surjective, which concludes the proof.
    \end{proof}
    \subsection*{Step 2: Solving the problem locally}
    In this step we prove the analogue result to Theorem \ref{thm chi2 torsion} for certain completions of $K$.
    \begin{prop}\label{prop local solution 1}
        Let $l/k$ a purely unramified finite extension. Set $K_0 := k(\!(t)\!)$ and $L_0 := l(\!(t)\!)$. Let $Z$ be a proper $K_0$-variety. Then the quotient
        \[ \mathrm{K}_{d+1}(K_0) \big/ \left\langle \mathrm{N}_{d+1}(L_0/K_0), \mathrm{N}_{d+1}(Z/K_0)\right\rangle\]
        is $\chi_{k}(Z,E)$-torsion for every coherent sheaf $E$ on $Z$.
    \end{prop}
    \begin{proof}
        We proceed by induction of $d$. The case $d=1$ corresponds to \cite[Proposition 3.5]{DiegoLuco2024KKp-adicFunction}. For each proper $K_0$-scheme $X$ we denote by $n_X$ the exponent of the group 
        \[\mathrm{K}_{d+1}(K_0)\big/\left\langle \mathrm{N}_{d+1}(L_0/K_0), \mathrm{N}_{d+1}(X/K_0)\right\rangle.\] 
        We say that $X$ satisfies property $(\mathrm{P})$ if it admits an irreducible, regular, proper, flat model over $\mathcal{O}_{k}$. It is enough to check assumptions $(1), (2)$ and $(3)$ from \cite[Proposition 2.1]{Wittenberg2015KKQp}. \par
        Assumption $(1)$ is obvious. Assumption $(3)$ is a direct consequence of Gabber and de Jong's theorem, see \cite[Introduction Theorem 3]{TravauxGabber}. It remains to check assumption $(2)$. Consider a proper $K_0$-scheme $Z$ together with an irreducible, regular, proper, flat model $\mathcal{Z}$ over $\mathcal{O}_k$. Denote the special fibre of $\mathcal{Z}$ by $Y$. Let $m$ be the multiplicity of $Y$ and $D$ an effective divisor on $\mathcal{Z}$ such that $Y=mD$. \par 
        Just as in the proof of \cite[Proposition 3.5]{DiegoLuco2024KKp-adicFunction} we have an exact sequence
        \begin{equation} \label{exseq residue in Ktheory}
            0 \to \frac{\mathrm{U}_{d+1}(K_0)}{\mathrm{U}_{d+1}(K_0) \cap \mathrm{N}_{d+1}(Z/K_0)} \to \frac{\mathrm{K}_{d+1}(K_0)}{\mathrm{N}_{d+1}(Z/K_0)} \to \frac{\mathrm{K}_{d}(k)}{\partial(\mathrm{N}_{d+1}(Z/K_0))} \to 0.
        \end{equation}
        Moreover,
        \begin{enumerate}
            \item[(a)] Since $k$ satisfies the $C_0^{d+1}$ property we can follow the proof of \cite[Lemma 4.4]{Wittenberg2015KKQp} to prove that the leftmost group in \eqref{exseq residue in Ktheory} is killed by $m$
            \item[(b)] The proof of \cite[Lemma 4.5]{Wittenberg2015KKQp} still holds in this context. Then we have $\partial(\mathrm{N}_{d+1}(Z/K_0)) = \mathrm{N}_{d}(Y/k) = \mathrm{N}_{d}(D/k)$.
            \item[(c)] By the induction hypothesis we know that 
            \[ \mathrm{K}_{d}(k) \big/ \langle \mathrm{N}_{d}(l/k), \mathrm{N}_{d}(D/k) \rangle\]
            is killed by $\chi(D,\mathcal{O}_D)$. 
        \end{enumerate}
        Let $\alpha \in \mathrm{K}_{d+1}(K_0)$. Then thanks to fact (c) we have 
        \[\chi(D,\mathcal{O}_D) \cdot \partial(\alpha)  \in \left\langle \mathrm{N}_{d}(l / k), \mathrm{N}_{d}(D/k) \right\rangle.\] 
        From the surjectivity of the residue map, fact (b) and the compatibility of the residue map with the norm map we can find $\beta \in \mathrm{N}_{d+1}(Z/K_0)$ and $\gamma \in \mathrm{N}_{d+1}(L_0 / K_0)$ such that $\chi(D,\mathcal{O}_D) \cdot \alpha - (\beta + \gamma) \in \mathrm{U}_{d+1}(K_0)$. We deduce that 
        \[ m \chi(D,\mathcal{O}_D) \cdot \alpha \in \left\langle\mathrm{N}_{d+1}(L_0 / K_0), \mathrm{N}_{d+1}(Z/K_0)\right\rangle\]
        using fact (a). Finally, we conclude noting that $m \chi(D, \mathcal{O}_D) = \chi(Z, \mathcal{O}_Z)$ thanks to \cite[Proposition 2.4]{EsnaultLevineWittenberg2015Index}.
    \end{proof}

    \subsection*{Step 3: Globalising the local field extensions.}
    The proof of \cite[Proposition 3.6]{DiegoLuco2024KKp-adicFunction} proves the following slightly more general result.
    \begin{prop}\label{prop globalising local extensions}
        Let $F$ be a Hilbertian field, $Z$ a smooth geometrically integral $F$-variety and $T$ a finite set of distinct discrete valuations on $K$. Fix a separable finite extension $F'/F$ and for each $w \in T$ a finite separable extension $M^{(w)}/F_w$ such that $Z(M^{(w)}) \neq \emptyset$, where $F_w$ represents the completion of $F$ with respect to $w$. Then there exists a finite extension $M/K$ satisfying the following properties:
        \begin{enumerate}
            \item $Z(M)\neq \emptyset,$
            \item For each $w \in T$ there exists a $K$-embedding $M \hookrightarrow M^{(w)}$, and
            \item The extension $F'/F$ and $M/F$ are linearly disjoint.
        \end{enumerate}
    \end{prop}
    
    \subsection*{Step 4: Computation of a Tate-Shafarevich group} 
    We can prove an analogue of \cite[Proposition 3.7]{DiegoLuco2024KKp-adicFunction}  exactly with the same proof. We will state it here for ease of reference.
    \begin{prop}\label{prop horrible}
        Let $r \geq 2$ be an integer and $L, K_1 \cdots, K_r$ be finite extensions of $K$ in $\overline{K}$. Consider the composite $K_{I} :=K_1 \dots K_r$ and $K_{\hat{i}}:= K_{1} \dots K_{i-1} K_{i+1} \cdots K_r$ for every $i\in \{1, \cdots r \}$. Denote by $n$ the degree of $L/K$. Consider the Galois module $\hat{T}$ defined by the exact sequence
        \[0 \to \mathbf{Z} \to \mathbf{Z}[E/K] \to \hat{T} \to 0,\]
        where $E:= L \times K_1 \times \cdots \times K_r$. Given two positive integers $m$ and $m'$, make the following assumptions:
        \begin{itemize}
            \item[(LD1)] The Galois closure of $L/K$ and the extension $K_I/K$ are linearly disjoint.
            \item[(LD2)] For each $i \in \{1, \cdots , r \}$ the fields $K_i$ and $K_{\hat{i}}$ are linearly disjoint over $K$.
            \item[(H1)] The restriction map \[\Sha^2(K,\hat{T}) \to \Sha^2(L,\hat{T}) \oplus \Sha^2(K_I,\hat{T})\] is injective.
            \item[(H2)] The restriction map \[\mathrm{res}_{L K_I /K_I}: \Sha^2(K_I,\hat{T}) \to \Sha^2(LK_I,\hat{T})\] is surjective and the kernel is $m$-torsion.
            \item[(H3)] For each $i$ the restriction maps \[\mathrm{res}_{L K_i /K_i}: \Sha^2(K_i,\hat{T}) \to\Sha^2(LK_i,\hat{T}) \qquad \mathrm{res}_{L K_{\hat{i}} /K_{\hat{i}}}: \Sha^2(K_{\hat{i}},\hat{T}) \to\Sha^2(LK_{\hat{i}},\hat{T}) \] are surjective.
            \item[(H4)] For each finite extension $L'$ of $L$ contained in the Galois closure $L/K$, the kernel of the restriction map \[\mathrm{res}_{L' K_I /L}: \Sha^2(L',\hat{T})\to\Sha^2(L'K_I,\hat{T})\] is $m'$-torsion.
        \end{itemize}
        Then $\overline{\Sha^2(K, \hat{T})}$ is $((m \vee m') \wedge n)$-torsion.
    \end{prop}
    \begin{lemma}\label{lemma torsion of kernel and surj shas}
        Let $l/k$ be a finite purely unramified extension of degree $n$ and set $L := lK$. The restriction map $\mathrm{res}_{L/K}: \Sha^2(K, \mathbf{Z}) \to \Sha^2(L, \mathbf{Z})$ is surjective and its kernel is $(i_{0}(C) \wedge n)$-torsion.
    \end{lemma}
    \begin{proof}
        After replacing every use of \cite[Corollary 2.9]{Kato1986Hasse2dim} by Proposition \ref{thm computation of Sha}, the same proof as in \cite[Lemma 3.12]{DiegoLuco2024KKp-adicFunction} works.
    \end{proof}

    \subsection*{Step 5: Proving the result for smooth varieties} 
    We now prove the statement of Theorem \ref{thm chi2 torsion}  for smooth varieties, but first we need a small lemma.
    \begin{lemma} \label{lemma residue and norms}
        Let $F_0$ be field of characteristic zero with $\mathrm{cd} (F_0) \leq d$ and $E_0/F_0$ a finite extension. Fix $E := E_0(\!(t)\!)$ and $F = F_0(\!(t)\!)$. Then the residue induces an isomorphism
        \[\partial: \mathrm{K}_{d}(F) \,/ \,\mathrm{N}_{d}( E/F) \to \mathrm{K}_{d-1}(F_0)\,/\,\mathrm{N}_{d-1}(E_0 / F_0).\]
    \end{lemma}
    \begin{proof}
        From the definition of the residue we can see directly that it is surjective. Since the residue map is compatible with norms and $\mathrm{U}_{d}(F)$ is by definition the kernel of the residue map, we can identify the kernel of $\partial$ with $\mathrm{U}_{d}(F) \big/ \mathrm{U}_{d}(F) \cap \mathrm{N}_{d}( E/F)$. We have the following commutative diagram with exact rows
        \begin{equation} \label{diag norm and units in K-theory}
            \begin{tikzcd}
                0 \ar[r] & \mathrm{U}_{d}^1(E) \ar[r] \ar[d, "\mathrm{N}_{E/F}"] & \mathrm{U}_{d}(E) \ar[r] \ar[d, "\mathrm{N}_{E/F}"] & \mathrm{K}_{d}(E_0) \ar[r] \ar[d, " \mathrm{N}_{E_0/F_0}"] & 0 \\
                0 \ar[r] & \mathrm{U}_{d}^1(F) \ar[r]  & \mathrm{U}_{d}(F) \ar[r]  & \mathrm{K}_{d}(F_0) \ar[r] & 0.
            \end{tikzcd}
        \end{equation}
        Since $\mathrm{U}_{d}^1(F)$ is divisible and the composition $\mathrm{U}_{d}^1(F) \to \mathrm{U}_{d}^1(E) \xrightarrow{\mathrm{N_{E/F}}} \mathrm{U}_{d}^1(F)$ is multiplication by $[E:F]$, the left-most vertical arrow in \eqref{diag norm and units in K-theory} is surjective. On the other hand, $\mathrm{N}_{E_0/F_0}$ is surjective because $F_0$ has the $C_0^{d}$. A direct diagram chase allows us to conclude that $ \mathrm{U}_{d}(F) \subseteq \mathrm{U}_{d}(F) \cap \mathrm{N}_{d}( E/F)$.
    \end{proof}
    This lemma generalises to several extensions as follows.
    \begin{prop}\label{prop generate residue implies generate}
        Let $F_0$ be a field of characteristic zero with $\mathrm{cd} (F_0) \leq d$ and $I$ a set of indices. Let $\{E_{0,i}\}_{i\in I}$ be a family finite extensions of $F_0$. For every $i \in I$ fix $E_i = E_{0,i}(\!(t)\!)$ and $F = F_0(\!(t)\!)$. If the subgroups $\mathrm{N}_{d-1}(E_{0,i}/F_0)$ generate $\mathrm{K}_{d-1}(F_0)$, then
        \[\mathrm{K}_d(F)  = \left\langle \mathrm{N}_d\left(E_i/F\right) \; | \; i \in I \right\rangle. \]
    \end{prop}
    \begin{proof}
        Consider $\alpha \in \mathrm{K}_{d}(F)$. By hypothesis there exists a finite subset $J \subseteq I$ and elements $\beta_j \in \mathrm{K}_{d-1}(E_{0,j})$ for $j \in J$ such that 
        \[\partial(\alpha) = \prod_{j \in J} \mathrm{N}_{E_{0,j}/F_0}(\beta_j).\]
        We prove by induction on the cardinality of $J$ that $\alpha$ belongs to $\left\langle \mathrm{N}_d\left(E_j/F\right) \; | \; j \in J \right\rangle$. The case $\#J =1$ follows from Lemma \ref{lemma residue and norms}. Since the residue is surjective we can find for every $j \in J$ an element $\tilde{\beta}_j \in \mathrm{K}_d(E_j)$ such that $\partial(\tilde{\beta}_j) = \beta_j$. Fix an index $j_0 \in J$. Since the norm mapsare compatible with the residue maps, we have
        \[\partial \left(\alpha - \mathrm{N}_{E_{j_0}/K}(\tilde{\beta}_{j_0})\right) \in \left\langle \mathrm{N}_{d-1}\left(E_{j,0}/F_0\right) \; | \; j \in J\setminus \{j_0\} \right\rangle.\] 
        Then by the induction hypothesis $\alpha - \mathrm{N}_{E_{j_0}/K}(\tilde{\beta}_{j_0})$ belongs to $\left\langle \mathrm{N}_d\left(E_j/F\right) \; | \; j \in J\setminus \{j_0\} \right\rangle.$ Then we conclude that $\alpha \in \left\langle \mathrm{N}_d\left(E_j/F\right) \; | \; j \in J \right\rangle$
    \end{proof}
    \begin{thm}\label{thm chisquare tosion smooth}
            Let $l/k$ be a finite purely unramified extension and set $L:= lK$. Let $Z$ be a proper smooth integral $K$-variety. Then the quotient
            \[ \mathrm{K}_{d+1}(K) \, \big/ \, \left\langle \mathrm{N}_{d+1}(L/K), \mathrm{N}_{d+1}(Z/K) \right\rangle\]
            is $\chi_{k}(Z,E)^2$-torsion for any coherent sheaf $E$ on $Z$.
        \end{thm}
    \begin{proof}
        Let $x \in \mathrm{K}_{d+1}(K)$. First observe that we can assume that $Z$ is geometrically integral by a restriction-corestriction argument. We may also assume that $i_0(C)=1$ thanks to Proposition \ref{prop reduction to index 1}. \par 
        Let $S$ be the finite set of places $v \in C^{(1)}$ such that $\partial_v(x) \neq 0$. Given a prime number $\ell$, amplness of the field $k$ ensure the existence of a point $w_{\ell} \in C^{(1)} \setminus S$ such that $[k(w_{\ell}):k]_0$ is prime to $\ell$, for a definition of this notion see \cite{Pop2014AmpleFields}. Moreover, by Proposition \ref{prop local solution 1} we know that
        \begin{equation}\label{eq locally chi torsion}
            \chi_{K}(Z,E) \cdot \mathrm{K}_{d+1}(K_{w_{\ell}}) \subseteq \left\langle \mathrm{N}_{d+1}(L_{w_{\ell}}/K_{w_{\ell}}), \mathrm{N}_{d+1}(Z_{w_{\ell}}/K_{w_{\ell}})\right\rangle
        \end{equation} 
        Just as in \cite[Lemma 3.14]{DiegoLuco2024KKp-adicFunction} we need the following lemma
        \begin{lemma}\label{lemma inequality val of chi}
            Keep the notation from Theorem \ref{thm chisquare tosion smooth} and let $n= [l:k]$. Denote by $v_{\ell}$ the $\ell$-adic valuation. If $v_{\ell}(n) \gneq v_{\ell}(\chi_K(Z,E))$, then there exists a finite field extension $M^{(w_{\ell})}$ of $K_{w_{\ell}}$ such that $Z(M^{(w_{\ell})}) \neq\emptyset$ and $v_{\ell}([m^{(w_{\ell})} : k(w_{\ell})]_0) \leq v_{\ell}(\chi_K(Z,E))$ where $m^{(w_{\ell})}$ is the first residue field of $M^{(w_{\ell})}$. 
        \end{lemma}

       \noindent Note that for every $v \in C^{(1)} \setminus S$ we have 
        \[x_v \in \mathrm{N}_{d+1}(L_v/K_v)\]
        thanks to Proposition \ref{lemma residue and norms}. For every $v \in S$, Proposition \ref{prop local solution 1} we can find finite extensions $M_1^{(v)}, \cdots , M_{r_v}^{(v)}$ of $K_v$ such that for every $i \in \{1, \cdots , r_v\}$ we have $Z(M_i^{(v)}) \neq \emptyset$ and 
        \[ \chi(Z,E) \cdot x \in \left\langle \mathrm{N}_{d+1}(L_v/K_v) ,  \mathrm{N}_{d+1}(M_i^{(v)}/K_v) \; | \; i \in \{1, \cdots , r_v\} \right\rangle.\]
        By repeatedly applying Proposition \ref{prop globalising local extensions} we can find for every $v \in S$ and $i \in \{1, \cdots , r_v\}$ a finite extension $K_i^{(v)}$ such that 
        \begin{enumerate}
            \item $Z(K_i^{(v)}) \neq \emptyset$,
            \item there is a $K$-embedding $K_i^{(v)} \hookrightarrow M_i^{(v)}$,
            \item for every prime $\ell$ such that $v_{\ell}(n) \gneq v_{\ell}(\chi(Z,E))$ there exists a $K$-embedding $K_i^{(v)} \hookrightarrow M^{(w_{\ell})}$ where $M^{(w_{\ell})}$ is the extension constructed in Lemma \ref{lemma inequality val of chi}, and
            \item for each pair $(v_0, i_0)$ the field $K_{i_0}^{(v_0)}$ is linearly disjoint from \[ L_n \cdot \prod_{(v,i) \neq (v_0,i_0)} K_i^{(v)}\] where $L_n$ is the composite of all cyclic locally trivial extensions whose degree is divides $n$. Note that $L_n$ is finite because $\Sha^1(K,\mathbf{Z}/n)$ is a finite group, see \cite[Proposition 2.6]{Diego2015LGPrincipleHigherLocal}.
        \end{enumerate}
        Fix $E=L \times \prod_{i,v} K^{(v)}_i$ and $K_I = \prod_{i,v} K^{(v)}_i$. Let $\hat{T}$ be the Galois module defined by the exact sequence
        \[ 0 \to \mathbf{Z} \to \mathbf{Z} [E/K] \to \hat{T} \to 0.\]
        In the same way as in the proof of \cite[Theorem 3.13]{DiegoLuco2024KKp-adicFunction} we can check the hypothesis of Proposition \ref{prop horrible} with $m = \chi(Z,E)$ and 
        \[ m' = \# \ker \left(\Sha^2(L,\mathbf{Z}) \to \Sha^2(LK_I, \mathbf{Z}) \right).\]
        Then by Proposition \ref{prop horrible} the group $\overline{\Sha^2(K,\hat{T})}$ is $((m\vee m') \wedge n)$-torsion. Since $K_I$ and $L_n$ are linearly disjoint, $m\wedge n =1$ and $(m\vee m') \wedge n = m \wedge n$. Set $\check{T} := \mathrm{Hom}_K \; (\hat{T},\mathbf{Z})$ and $T:= \check{T} \otimes \mathbf{Z}(d+1)$. By Poitou-Tate duality we deduce that $\Sha^{d+2}(K,T)$ is $m$-torsion. \par 
        We can interpret $x$ as an element of $\mathrm{H}^{d+2}(K,T)$ thanks to Lemma \ref{lemma norms Ktheory and motivic}. By construction of $T$ we know that $m x$ belongs to $\Sha^{d+2}(K,T)$ and since this group is $m$-torsion, $m^2 \cdot x =0 \in \Sha^{d+1}(K,T)$. Which again by Lemma \ref{lemma norms Ktheory and motivic} means that $m^2 \cdot x $ belongs to 
        \[\left\langle \mathrm{N}_{d+1}(L/K), \mathrm{N}_{d+1}(K_i^{(v)}/K) \; |\; v \in S ,\: i \in \{1, \cdots ,r_v \} \right\rangle\]
        which is a subset of 
        \[\left\langle  \mathrm{N}_{d+1}(L/K), \mathrm{N}_{d+1}(Z/K) \right\rangle\]
        by construction of $K_i^{(v)}$.
    \end{proof}
    \subsection*{Step 6: The general case}
        The general case of Theorem \ref{thm chi2 torsion} can be deduced from the smooth case using the same \textit{dévissage} methods used in \cite{DiegoLuco2024KKp-adicFunction}; namely the following \textit{dévissage} principle can be applied in our context, see \cite[Proposition 3.15]{DiegoLuco2024KKp-adicFunction}
        \begin{prop}
        Let $F$ be a field and $r$ a positive integer. Let $(\mathrm{P})$ be a property of $F$-varieties. Suppose that for every proper $F$-variety we are given an integer $m_X$. Assume that 
        \begin{enumerate}
            \item For every morphism of proper $F$-varieties $Y\to X$, the integer $m_X$ divides $m_Y$,
            \item for every proper $F$-variety satisfying $(\mathrm{P})$, the integer $m_X$ divides $\chi_F(X,\mathcal{O}_X)^r$,
            \item for every proper integral $F$-scheme $X$ there exists a proper $F$-scheme $Y$ satisfying $(\mathrm{P})$ and an $F$-morphism $Y \to X$ with generic fibre $Y_{\eta}$ such that $m_X$ and $\chi_{F(X)}(Y_{\eta},\mathcal{O}_{Y_{\eta}})^r$ are coprime.
        \end{enumerate}
        Then for every proper $F$-variety $X$ and coherent sheaf $E$ on $X$, the integer $m_X$ divides $\chi_K(X,E)^r$. 
        \end{prop}
        \noindent This proposition is a refinement of \cite[Proposition 2.1]{Wittenberg2015KKQp}.
    \subsection*{The \texorpdfstring{$C_{d+1}^{d+1}$}{Kato Kuzumaki} property}
    As an application of Theorem \ref{thm chi2 torsion} we get the following result.
    \begin{thm}\label{thm Cd1d1}
        Let $k$ be a $d$-local $k$ field such that $k_1$ is a $p$-adic field. Then the field $K= k(x)$ has the $C_{d+1}^{d+1}$ property.
    \end{thm}
    \begin{proof}
        Let $L$ be a finite field extension of $K$ and $Z$ a hypersurface of $\mathbf{P}^n_{L}$ of degree $m$ such that $m^{d+1} \leq n$. Note that $L$ is the function field of a smooth projective geometrically integral curve $C$ over a $d$-local field $l$. Since $L^{pnr} = l^{pnr}L$ satisfies the $C_{d+1}$ property, we know that there exists a finite purely unramified extension $l'/l$ such that $Z(l') \neq \emptyset$. Set $L' := l' K$. Now we can apply Theorem \ref{thm chi2 torsion} to $Z$ and deduce that 
        \[ \mathrm{K}_{d+1}(L)  \big/  \left\langle \mathrm{N}_{d+1}(L'/L) , \mathrm{N}_{d+1}(Z/K) \right\rangle = \mathrm{K}_{d+1}(L)  \big/   \mathrm{N}_{d+1}(Z/K) \]
        is $\chi(Z, \mathcal{O}_Z)^2$-torsion. But since $m\leq n$ we can apply \cite[Theorem III 5.1]{Hartshorne1983AlgebraicGeometry} and the exact sequence
        \[0 \to \mathcal{O}_{\mathbf{P}^n_K}(-d) \to \mathcal{O}_{\mathbf{P}^n_K} \to \iota_* \mathcal{O}_Z \to 0\]
        where $\iota: Z \to \mathbf{P}^n_K$ stands for the closed immersion, to deduce that $\chi_K(Z,\mathcal{O}_Z) = 1$.
    \end{proof} 
\section{About the \texorpdfstring{$C_{j+1}^{d+1}$}{Kato Kuzumaki} property}
The setting is the same as the previous section: $k$ is a $d$-local field such that $k^{(d-1)}$ is a $p$-adic field and $K$ is the function field of a smooth projective geometrically integral curve $C$ defined over $k$. In this section we generalise the methods from \cite[Section 4]{DiegoLuco2024KKp-adicFunction} to get results in the direction of the $C_{j+1}^{d+1}$ property for some $j$ smaller than $d$. \par 
For the rest of this section, fix an integer $j \in \{1 ,\cdots , d \}$. Let $l/k$ be a finite extension totally $j$-ramified of prime degree $\ell$. We define $\mathcal{E}^0_{l/k}$ as the set of finite $j$-ramified subextensions of $l^{pnr}/k$. Set
\[ \mathcal{E}_{l/k} := \{ K'/K \, | \, K' = k'K \text{ with } k' \in \mathcal{E}_{l/k} \}.\]
In the rest of the section we use extensively the following observation: for every pair of extensions $k'$ and $k''$ in $\mathcal{E}^0_{l/k}$ such that $k' \subseteq k''$ the extension $k'' / k'$ is purely unramified.
\begin{thm}\label{thm Elk generated Ktheory}
    Assume that $i_{\leq j}(C) = 1$. Let $\ell$ be a prime number and $l/k$ a totally $j$-ramified extension of degree $\ell$. Then we have 
    \[ \mathrm{K}_{d+1}(K) = \left\langle \mathrm{N}_{d+1}(K'/K) \, | \, K' \in \mathcal{E}_{l/k} \right\rangle.\]
\end{thm}
\subsection*{Step 1: Solving the problem locally}
The following proposition is valid for any $d$-local field regardless the level where the characteristic gets mixed.
\begin{prop}\label{prop generators Ktheory dlocal}
    Let $t_1,\cdots, t_d$ be a system of parameters of $k$ and $n \in \mathbf{N}$. Then $\mathrm{K}_n(k)$ is generated by symbols of the form $\{t_1, \cdots t_m, \lambda_{m+1}, \cdots, \lambda_{d}\}$ where $m \in \{1,\cdots, m\}$, and for every $j \in \{m+1, \cdots, d\}$ the element $\lambda_j$ is in $\mathcal{O}_k^{\times}$.
\end{prop}
\begin{proof}
    This is a direct computation given the fact that $\mathcal{O}_k^{\times}$ together with $t_1, \cdots ,t_n$ generates $k^{\times}$.
\end{proof}

\begin{lemma}\label{lemma not totally j}
    Let $\ell$ be a prime number and $l/k$ a finite Galois totally $j$-ramified extension of degree $\ell$. Let $m/k$ be a finite purely ramified and $j$-ramified extension such that $ml/m$ is purely unramified. Then there exists $k' \in \mathcal{E}_{l/k}^0$ contained in $m$.
\end{lemma}
\begin{proof}
    If $ml/m$ is trivial, $l$ is contained in $m$ and there is nothing to prove. Then we may and do assume that $ml/m$ has degree $\ell$. Let $k_{\ell}$ be the purely unramified extension of $k$ of degree $\ell$ and set $l_{\ell} := lk_{\ell} $. The extension $l_{\ell}/k$ is Galois with Galois group $(\mathbf{Z}/\ell )^2$. Note that $ml$ contains $l_{\ell}$ because $ml/m$ is purely unramified of degree $\ell$.  Note that we have the following inequalities
    \[[m:k] [l_{\ell}:k] = \ell^2 [m:k] \gneq \ell [m:k] = [ml:k] = [ml_{\ell}:k].\]
    Denote by $k'$ the intersection $m \cap l_{\ell}$. The previous inequality imply that $[k:k] \gneq 1$ and as a subextension of $l_{\ell}$ its degree can only be $\ell$ or $\ell^2$, but it cannot be $\ell^2$ because that would imply $l_{\ell} \subseteq m$. Since $m/k$ is purely ramified and $j$-ramified, we conclude that $k' \in \mathcal{E}_{l/k}^0$.
\end{proof}
\begin{prop}\label{prop Elk generates locally}
    Let $\ell$ be a prime number and $l/k$ be a finite extension of degree $\ell$ totally $j$-ramified for $j \in \{1, \cdots, d\}$. Then for every $v \in C^{(1)}$ we have 
    \begin{equation}\label{eq generation K locally}
        \mathrm{K}_{d+1}(K_v) =\left\langle \mathrm{N}_{d+1} (K' \otimes_K K_v /K_v) \; |\; K' \in \mathcal{E}_{l/k} \right\rangle.
    \end{equation}
\end{prop}
\begin{proof}
    Thanks to Proposition \ref{prop generate residue implies generate} we only need to prove that $\mathrm{K}_{d}(k(v))$ is generated by the norms coming from $k' \otimes_k k(v)$ with $k' \in \mathcal{E}^0_{l/k}$. We prove this statement by induction on $d$. The case $d=1$ is treated in \cite[Proposition 4.4]{DiegoLuco2024KKp-adicFunction}.  \par
    Let us consider the case $j \neq d$. Then by the induction hypothesis, we see that 
    \[\mathrm{K}_{d-1}(k(v)^{(1)}) = \left\langle \mathrm{N}_{d-1} (m \otimes_{k^{(1)}} k(v)^{(1)} /k(v)^{(1)}) \; |\; m \in \mathcal{E}_{l^{(1)}/k^{(1)}}^0 \right\rangle. \]
    We conclude this case by applying Proposition \ref{prop generate residue implies generate} and noting that an extension $m/k_{d-1}$ is in $\mathcal{E}^0_{l^{(1)}/k^{(1)}}$ if and only if there exists an extension $k' \in \mathcal{E}^0_{l/k}$ whose first residual extension $k'^{(1)}/k^{(1)}$ is equal to $m$. \par
    Assume that $j=d$. We distinguish four cases 
    \begin{enumerate}
        \item $l$ is contained in $k(v)$,
        \item $lk(v)/k(v)$ is purely unramified of degree $\ell$,
        \item $lk(v)/k(v)$ is totally $i$-ramified of degree $\ell$ for $i\in \{1 , \cdots, d-1 \}$, or
        \item $lk(v)/k(v)$ is totally $d$-ramified of degree $\ell$.
    \end{enumerate}
    Case 1 is direct because $l \otimes k(v) \simeq k(v) \times \cdots \times k(v)$. Let us consider case 2. Note that $k(v)/k$ is $d$-ramified because otherwise $lk(v)/k(v)$ would necessarily be $d$-ramified. Let $k(v)_{pnr}/k$ be the maximal purely unramified subextension of $k(v)$. We can apply Lemma \ref{lemma not totally j} because the extension $k(v)/k(v)_{pnr}$ is a purely ramified $d$-ramified extension such that $lk(v)/k(v)_{pnr}$ is purely unramified. Then there exists a finite extension $k' \in \mathcal{E}^0_{lk(v)_{pnr}/k(v)_{pnr}} \subseteq \mathcal{E}^0_{l/k}$ contained in $k(v)$.\par 
    Let us consider case 3. As in case 2, the extension $k(v)/k$ must be $d$-ramified. Since the statement was already proved for $j\neq d$ we know that 
    \begin{equation} \label{eq Elk generates locally case 2}
        \mathrm{K}_d(k(v)) =\left\langle \mathrm{N}_{d} (k' /k(v)) \; |\; k' \in \mathcal{E}_{lk(v)/k(v)}^0 \right\rangle.
    \end{equation}
    Let $k'$ be an extension in $\mathcal{E}^0_{l k(v)/k(v)}$. Denote by $k'_{pnr}/k$ the maximal purely unramified subextension of $k'/k$. We can apply Lemma \ref{lemma not totally j} because $k'/k'_{pnr}$ is a purely ramified $d$-ramified extension such that $k'l/k'$ is purely unramified. Then we can find an extension $k'' \in \mathcal{E}^0_{lk'_{pnr}/k_{pnr}'} \subseteq \mathcal{E}_{l/k}$ contained in $k'$. Since $k''$ is contained in $k'$ we have $\mathrm{N}_d(k'/k(v)) \subseteq \mathrm{N}_d(k''k(v)/k(v))$ where $k''k(v)$ is the composite of $k''$ and $k(v)$ inside $k'$. In particular, $\mathrm{N}_d(k'/k(v)) \subseteq \mathrm{N}_d(k'' \otimes_{k(v)} k(v) /k(v))$. We conclude this case using \eqref{eq Elk generates locally case 2}. \par 
    Finally, let us consider case 4. Since $d\geq 2$ we can find a system of parameters $t_1, \cdots, t_d$ for $k(v)$ such that $t_d$ belongs to $k$ and $l = k(\sqrt[\ell]{t_d})$, see \cite[Proposition II 5.12]{Lang1994AlgebraicNumberTheory}. In particular, we have $lk(v) = k(v)(\sqrt[\ell]{t_d})$. Clearly $t_d$ belongs to $\mathrm{N}_{1}(lk(v)/ k(v))$. Note that for every $u \in \mathcal{O}^{\times}_{k(v)}$ the extension $k(\sqrt[\ell]{ut_d})/k$ belongs to $\mathcal{E}^0_{l/k}$. Then any symbol $\alpha$ in $\mathrm{K}_d(k(v))$ that has either $t_d$ or an element of $\mathcal{O}_{k(v)}^{\times}$ as an entry in 
    \[\left\langle \mathrm{N}_{d} (k' \otimes_k k(v) /k(v)) \; |\; k' \in \mathcal{E}_{l/k}^0 \right\rangle.\]
    We conclude case 4 applying Proposition \ref{prop generators Ktheory dlocal}.
\end{proof}
\subsection*{Step 2: Computation of a Tate-Shafarevich group}
In this step we compute a Tate-Shafarevich group that controls the defect of the \textit{multi-normic} principle for extensions in $\mathcal{E}_{l/k}$.
\begin{prop}
    Assume that $i_{\leq j}(C) = 1$. Let $l/k$ be a finite Galois extension totally $j$-ramified of prime degree $\ell$. Given $K_1, \cdots, K_r \in \mathcal{E}_{l/k}$ such that $K_1$ and $K_2$ are linearly disjoint over $K$, consider the Galois module $\hat{T}$ defined by the following exact sequence  
    \begin{equation}
        0 \to \mathbf{Z} \to \mathbf{Z}[E/K] \to \hat{T} \to 0,
    \end{equation}
    where $E := K_1 \times \cdots \times K_r$. Then $\Sha^2(K,\hat{T})$ is divisible. 
\end{prop}
\begin{proof}
    We have a complex 
    \[ \Sha^2(K,\hat{T}) \xrightarrow{f_0} \Sha^2(K_1,\hat{T}) \oplus \Sha^2(K_2,\hat{T}) \xrightarrow{g_0} \Sha^2(K_1K_2,\hat{T})\]
    where $f_0$ and $g_0$ are defined as follows
    \begin{align*}
        f_0(\alpha) &= (\mathrm{res}_{K_1/K}(\alpha),\mathrm{res}_{K_2/K}(\alpha)) \\
        g_0(\alpha,\beta) &= \mathrm{res}_{K_1K_2/K_1}(\alpha)-\mathrm{res}_{K_1K_2/K_2}(\beta).
    \end{align*}
    \begin{lemma}
        $f_0$ is injective.
    \end{lemma}
    \begin{proof}
        We fix an algebraic closure $\overline{K}$ of $K$ and all fields are considered inside $\overline{K}$. Let $K_I$ be the Galois closure of the composite $K_1 \cdots K_r$ and $k_I/k$ the corresponding extension. Let $S \subseteq C^{(1)}$ be a finite family of points such that 
        \[ \mathrm{gcd} \left([k(v):k]_j \, | \, v \in S \right) =1. \] 
        For every $v \in S$ denote by $G_v \leq \mathrm{Gal}(K_I/K) = \mathrm{Gal}(k_I/k)$ the decomposition group of $v$ in $K_I/K$. We have the following identification of subgroups of $\mathrm{Gal}(k_I/ k)$
        \[G_v = \mathrm{Gal}(k_I/ k_I \cap k(v))\]
        Noting that for every $i \in \{j+1, \cdots, d\}$ the extension $k_I/k$ is not $i$-ramified, we deduce that the index of $G_v$ in $\mathrm{Gal}(k_I/k)$ divides $[k(v):k]_j$. A small exercise in group theory shows that a family of subgroups with coprime index must generate the group. For each $v \in S$ we choose a place $w$ of $K_I$ above $v$ and denote by $K_{I,v}$ the completion $K_{I,w}$. Thus, the restriction map 
        \[\mathrm{H}^2(K_I/K, \hat{T}) \to \prod_{v \in S} \mathrm{H}^2(K_{I,v}/K_v, \hat{T})\]
        is injective. By inflation-restriction we have the following commutative diagram with exact row
        \begin{equation*}
        \begin{tikzcd}
            0 \ar[r] & \mathrm{H}^2(K_I/K, \hat{T}) \ar[r] \ar[d] & \mathrm{H}^2(K, \hat{T}) \ar[r] \ar[d] & \mathrm{H}^2(K_I, \hat{T}) \ar[d] \\
            0 \ar[r] & \prod_{v \in S} \mathrm{H}^2(K_{I,v}/K_v, \hat{T}) \ar[r] & \prod_{v \in S} \mathrm{H}^2(K_v, \hat{T}) \ar[r]& \prod_{v \in S} \mathrm{H}^2(K_{I,v}, \hat{T}).
        \end{tikzcd}
    \end{equation*}
    We deduce that the morphism 
    \[ \ker \left(\mathrm{H}^2(K, \hat{T}) \to  \prod_{v \in S} \mathrm{H}^2(K_v, \hat{T})\right) \to \ker \left(\mathrm{H}^2(K_I, \hat{T}) \to  \prod_{v \in S} \mathrm{H}^2(K_{I,v}, \hat{T})\right) \]
    is injective. In particular, the restriction $\Sha^2(K,\hat{T}) \to \Sha^2(K_I,\hat{T})$ is also injective. We conclude because $\Sha^2(K,\hat{T}) \to \Sha^2(K_I,\hat{T})$ factors through $\Sha^2(K,\hat{T}) \to \Sha^2(K_1,\hat{T})$ and $\Sha^2(K,\hat{T}) \to \Sha^2(K_2,\hat{T})$. 
    \end{proof}
    Now we can follow the proof of \cite[Proposition 4.5]{DiegoLuco2024KKp-adicFunction} replacing every instance of \cite[Lemma 3.12]{DiegoLuco2024KKp-adicFunction} by Lemma \ref{lemma torsion of kernel and surj shas} and noting that $i_0(C) = 1$ since $i_0(C)$ divides $i_{\leq j}(C)$.
    \end{proof}
    \subsection*{Step 3: Conclusion of the argument}
    \begin{proof}[Proof of Theorem \ref{thm Elk generated Ktheory}]
        Consider $x \in \mathrm{K}_{d+1}(K)$. For every $v \in C^{(1)}$ we have 
    \[\mathrm{K}_{d+1}(K_v) =\left\langle \mathrm{N}_{d+1} (K' \otimes_K K_v /K_v) \; |\; K' \in \mathcal{E}_{l/k} \right\rangle.\]
    Hence, we can find $K_1, \cdots, K_r \in \mathcal{E}_{l/k}$ such that the class of $x$ in 
    \[\mathrm{K}_{d+1}(K) \, \big / \, \left\langle  \mathrm{N}_{d+1} (K_i /K) \, |  \, i \in \{1, \cdots, r \} \right\rangle\]
     belongs to the kernel of the restriction map towards 
    \[ \prod_{v \in C^{(1)}} \mathrm{K}_{d+1}(K_v) \, \big / \, \left\langle  \mathrm{N}_{d+1} (K_i \otimes_K K_v /K_v) \, | \, i \in \{1, \cdots, r \} \right\rangle.\]
    Moreover, up to enlarging the family $K_i$, we may assume that $K_1$ and $K_2$ are linearly disjoint. Consider the étale algebra $E:= K_1 \times \cdots \times K_r$ and the Galois module $\hat{T}$ defined by the following exact sequence
    \[0 \to \mathbf{Z} \to \mathbf{Z}[E/K] \to \hat{T} \to 0. \]
    Set $\check{T} := \mathrm{Hom}_K \;(\hat{T}, \mathbf{Z})$ and $T:= \check{T} \otimes \mathbf{Z}(d+1)$. Then under the identifications in Lemma \ref{lemma norms Ktheory and motivic} the class of $x$ belongs to $\Sha^{d+2}(K,T)$. On the other hand, the group $\Sha^2(K,\hat{T})$ is divisible and Poitou-Tate duality \cite[Theorem 0.1]{Diego2015LGPrincipleHigherLocal} implies that $\Sha^{d+2}(K,T)$ is trivial. Then $x$ belongs to $\left\langle  \mathrm{N}_{d+1} (K_i /K) \, |  \, i \in \{1, \cdots, r \} \right\rangle$, concluding the proof of Theorem \ref{thm Elk generated Ktheory}.
    \end{proof}

    \subsection*{Application to the \texorpdfstring{$C_{j+1}^{d+1}$}{Kato Kuzumaki} property}
    Theorem \ref{thm Elk generated Ktheory} allows us to prove the following result.
    \begin{thm}\label{thm iramchi torsion}
        Let $K$ be the function field of a smooth projective curve $C$ defined over a $d$-local field $k$ such that the $1$-local residue field of $k$ is $p$-adic. Let $l/k$ be a finite Galois extension, a number $j \in \{1,\cdots , d\}$ and set $L:=lK$. Assume that $l/k$ is not $i$-ramified for every $i\in \{j+1, \cdots , d\}$. Let $Z$ be a proper $K$-variety. Set $s_{l/k}$ to be the number of (not necessarily distinct) prime factors of the total ramification index of $l/k$. Then the quotient
        \[ \mathrm{K}_{d+1}(K) \, \big/ \, \left\langle \mathrm{N}_{d+1}(L/K) , \mathrm{N}_{d+1}(Z/K)\right\rangle\]
        is $i_{\leq j}^{ram}(C) \cdot \chi_K(Z,E)^{2 s_{l/k} +4}$-torsion for every coherent sheaf $E$ over $Z$.
    \end{thm}
    \begin{proof}
        We first assume that $i_{\leq j}(C) = 1$, and prove that the quotient in the statement is $\chi_K(Z,E)^{2s_{l/k} +2 }$-torsion for every coherent sheaf $E$ over $Z$. We proceed by induction on $s_{l/k}$. The case $s_{l/k} = 0$ is the content of Theorem \ref{thm chi2 torsion}. Assume that $s_{l/k}$ is greater than zero. Denote by $l_{pnr}$ the maximal subextension of $l/k$ such that $l_{pnr}/k$ is purely unramified, and set $L_{pnr} = l_{pnr}K$. By Theorem \ref{thm chi2 torsion} the quotient
        \[ \mathrm{K}_{d+1}(K) \, \big/ \, \left\langle \mathrm{N}_{d+1}(L_{pnr}/K) , \mathrm{N}_{d+1}(Z/K)\right\rangle\] 
        is $\chi_K(E,Z)^2$-torsion. Let $j' \in\{1, \cdots, j \}$ be the smallest index such that $l/l_{pnr}$ is $j'$-ramified. Denote by $l_{\leq j'}/l_{pnr}$ the maximal subextension of $l/l_{pnr}$ that not $i$-ramified for every $i \in \{j'+1,\cdots, d\}$. Since $l_{\leq j'}/l_{pnr}$ is a solvable extension, we can extract a subextension $m/l_{pnr}$ totally $j'$-ramified of prime degree $\ell$. Set $M = mK$. In one hand Theorem \ref{thm Elk generated Ktheory} tells us that
        \[\mathrm{K}_{d+1}(L_{pnr}) = \left\langle \mathrm{N}_{d+1}(K'/K) \, | \, K' \in \mathcal{E}_{m/l_{pnr}} \right\rangle.\]
        On the other hand, for each $k' \in \mathcal{E}^0_{m/l_{pnr}}$ the total ramification degree of $lk'/k'$ strictly divides that of $l/k$. Indeed, we can see that $[mk':k'] \cdot e_{tot}(lk'/k')$ divides $e_{tot}(l/k)$. Then $e_{tot}(l/k) =e_{tot}(lk'/k')$ would imply $m=k'$ which gives a contradiction because $e_{tot}(l/k) = \ell \cdot e_{tot}(l/m)$. So we can apply our induction hypothesis to deduce that for every $K' \in \mathcal{E}_{m/l_{pnr}}$ the quotient
        \[ \mathrm{K}_{d+1}(K') \, \big/ \, \left\langle \mathrm{N}_{d+1}(K'L/K') , \mathrm{N}_{d+1}(Z \otimes_K K'/K')\right\rangle\]
        is $\chi_{K}(Z,E)^{2 s_{l/k}}$-torsion. Putting these facts together we deduce that 
        \[\mathrm{K}_{d+1}(K) \, \big/ \, \left\langle \mathrm{N}_{d+1}(L/K) , \mathrm{N}_{d+1}(Z/K) \right\rangle\]
        is $\chi_K(Z,E)^{2s_{l/k} +2 }$-torsion. \par 
        We no longer assume $i_{\leq j}(C) =1$. Let $k_1, \cdots ,k_n$ be finite extensions of $k$ such that for every $ r \in \{1, \cdots , n \}$ the curve $C$ has a $k_r$-point, and
        \[ i_{\leq j}^{ram}(C) = \mathrm{gcd}\left( e_{\leq j}(k_r/k) \, | \, r \in \{1, \cdots , n\} \right).\]
        For each index $r$, denote by $k_{r,pnr}$ the maximal subextension of $k_r/k$ such that $k_{r,pnr}/k$ is purely unramified, and set $K_{r,pnr} := k_{r,pnr} K$. Thanks to Theorem \ref{thm chi2 torsion} we know that the quotient
        \[ \mathrm{K}_{d+1}(K) \, \big/ \, \left\langle \mathrm{N}_{d+1}(K_{r,pnr}/K) , \mathrm{N}_{d+1}(Z/K)\right\rangle\]
        is $\chi_K(Z,E)^2$-torsion for every $r$. For each $r$ let $m_r/k$ be the maximal subextension of $k_r/k$ such that $m_r/k$ is $i$-unramified for every $i \in \{j+1, \cdots ,d \}$, and set $M_r := m_rK$. Note that $k_{r,pnr}$ is contained in $m_r$ and  $[m_r:k] = [k_r:k]_j$. A restriction-corestriction argument shows that 
        \[ \mathrm{K}_{d+1}(K_{r,pnr}) / \mathrm{N}_{d+1}(M_{r}/K_{r,pnr})\]
        is $[m_r :k_{r,pnr}]$-torsion. Note that $[m_r: k_{r,pnr}] = e_{\leq j}(k_r/k)$. We deduce that the quotient
        \[\mathrm{K}_{d+1}(K) \, \big/ \, \left\langle \mathrm{N}_{d+1}(M_1/K), \cdots , \mathrm{N}_{d+1}(M_n/K) , \mathrm{N}_{d+1}(Z/K)\right\rangle\]
        is $i_{\leq j}^{ram}(C) \cdot \chi_K(Z,E)^2$-torsion. Since over $m_r$ we have $i_{\leq j}(C \otimes_k m_r) =1$ we can apply the first case to see that 
        \[\mathrm{K}_{d+1}(M_r) \, \big/ \, \left\langle \mathrm{N}_{d+1}(M_r L/M_r) , \mathrm{N}_{d+1}(Z \otimes_K M_i/M_i)\right\rangle\]
        is $\chi_K(Z,E)^{2 s_{l/k} +2}$-torsion. Then 
        \[\mathrm{K}_{d+1}(K) \, \big/ \, \left\langle \mathrm{N}_{d+1}(L/K) , \mathrm{N}_{d+1}(Z/K)\right\rangle\]
        is $i_{\leq j}^{ram}(C) \cdot \chi_K(Z,E)^{2 s_{l/k} +4}$-torsion.
    \end{proof}
    \noindent As corollaries, we get the following results about the $C_{j+1}^{d+1}$ property.
    \begin{cor}
        Let $K$ be the funciton field of a smooth projective curve $C$ defined over a $d$-local field $k$ such that the $1$-local residue field of $k$ is $p$-adic and $j \in \{0, \cdots , d \}$. Then, for every $m,n \geq 1$ and hypersurface $Z$ in $\mathbf{P}^{n}_K$ of degree $m$ such that $m^{j+1} \leq n$, the quotient $\mathrm{K}_{d+1}(K) / \mathrm{N}_{d+1}(Z/K)$ is $i_{\leq d-j}^{ram}(C)$-torsion
    \end{cor}
    \begin{proof}
        Let $Z$ be a hypersurface of $\mathbf{P}^n_k$ of degree $m$ with $m^{j+1} \leq n$. Let $k^{\leq d -j}/k$ be the maximal extension of $k$ that is not $i$-ramified for any $i \in \{d-j+1, \cdots,d\}$. We can combine Tsen's Theorem \cite[Theorem 2a]{Nagata1957NotesLang} and Lang's Theorem \cite[Theorem 12]{Lan1952QAlgClosure} to deduce that the field $k^{\leq d-j}(C)$ satisfies $C_{j+1}$. Then there exists a finite Galois extension $l/k$ such that $Z(lK)\neq \emptyset$ and $l/k$ is not $i$-ramified for every $i \in \{d-j+1,\cdots, d\}$. Theorem \ref{thm iramchi torsion} ensures that the quotient
        \[\mathrm{K}_{d+1}(K) \big/ \left\langle \mathrm{N}_{d+1}(lK) , \mathrm{N}_{d+1}(Z/K) \right\rangle = \mathrm{K}_{d+1}(K) \big/ \left\langle\mathrm{N}_{d+1}(Z/K) \right\rangle\]
        is $i_{\leq d-j}^{ram}(C) \cdot \chi_K(Z,E)^{2s_{l/k}+4}$-torsion for any coherent sheaf $E$. But since $m\leq n$ we can deduce from \cite[Theorem III 5.1]{Hartshorne1983AlgebraicGeometry} and the exact sequence
        \[0 \to \mathcal{O}_{\mathbf{P}^n_K}(-d) \to \mathcal{O}_{\mathbf{P}^n_K} \to \iota_* \mathcal{O}_Z \to 0\]
        where $\iota: Z \to \mathbf{P}^n_K$ stands for the closed immersion, that $\chi_K(Z,\mathcal{O}_Z) = 1$. Thus, $\mathrm{K}_{d+1}(K) \big/ \left\langle\mathrm{N}_{d+1}(Z/K) \right\rangle$ is $i^{ram}_{\leq d-j}(C)$-torsion.
    \end{proof}
    \begin{cor}
        Let $K$ be the funciton field of a smooth projective curve $C$ defined over a $d$-local field $k$ such that the $1$-local residue field of $k$ is $p$-adic and $j \in \{0,\cdots, d\}$. Assume that $i_{\leq d-j}^{ram} (C) = 1$ Then, for every $m,n \geq 1$ and hypersurface $Z$ in $\mathbf{P}^{n}_K$ of degree $m$ such that $m^{j+1} \leq n$, we have $\mathrm{K}_{d+1}(K)  =  \mathrm{N}_{d+1}(Z/K)$.
    \end{cor}
    \section*{Acknowledgemnts}
    I would like to thank Diego Izquierdo for his advice, suggestions and support. Without him this work would not have seen the light of day.
    \bibliographystyle{alpha}
    \bibliography{ref.bib}
\end{document}